\documentclass[12pt,oneside,reqno]{amsart}
\usepackage{}
\usepackage{amssymb}
\usepackage[colorlinks=true, linkcolor=blue, citecolor=blue]{hyperref}
\usepackage{bbm}
\usepackage{cases}
\usepackage{amsmath}
\usepackage{graphicx}
\usepackage{mathrsfs}
\usepackage{stmaryrd}
\usepackage{color}
\usepackage{soul}
\usepackage[dvipsnames]{xcolor}
\usepackage{amsfonts}
\usepackage{enumerate,amsmath,amssymb,amsthm}

\numberwithin{equation}{section}

\newcommand{\be}{\begin{eqnarray}}
\newcommand{\mE}{\end{eqnarray}}
\newcommand{\ce}{\begin{eqnarray*}}
\newcommand{\de}{\end{eqnarray*}}
\newtheorem{theorem}{Theorem}[section]
\newtheorem{lemma}[theorem]{Lemma}
\newtheorem{remark}[theorem]{Remark}
\newtheorem{definition}[theorem]{Definition}
\newtheorem{proposition}[theorem]{Proposition}
\newtheorem{example}[theorem]{Example}
\newtheorem{corollary}[theorem]{Corollary}

\def\e{{\mathrm{e}}}
\def\eps{\varepsilon}

\def\p{\partial}

\def\[{{\Big[}}
\def\]{{\Big]}}
\def\<{{\langle}}
\def\>{{\rangle}}
\def\({{\Big(}}
\def\){{\Big)}}

\def\bx{{\mathbf{x}}}

\def\dif{{\mathord{{\rm d}}}}

\def\no{\nonumber}
\def\={&\!\!=\!\!&}
\def\bt{\begin{theorem}}
\def\et{\end{theorem}}
\def\bl{\begin{lemma}}
\def\el{\end{lemma}}
\def\br{\begin{remark}}
\def\er{\end{remark}}

\def\bd{\begin{definition}}
\def\ed{\end{definition}}
\def\bp{\begin{proposition}}
\def\ep{\end{proposition}}
\def\bc{\begin{corollary}}
\def\ec{\end{corollary}}
\def\bx{\begin{example}}
\def\ex{\end{example}}

\def\cI{{\mathcal I}}

\def\cK{{\mathcal K}}
\def\cL{{\mathcal L}}

\def\cP{{\mathcal P}}

\def\cU{{\mathcal U}}
\def\cV{{\mathcal V}}
\def\cW{{\mathcal W}}

\def\mC{{\mathbb C}}

\def\mE{{\mathbb E}}
\def\mF{{\mathbb F}}

\def\mI{{\mathbb I}}

\def\mN{{\mathbb N}}

\def\mP{{\mathbb P}}

\def\mR{{\mathbb R}}

\def\sF{{\mathscr F}}

\def\sJ{{\mathscr J}}
\def\sK{{\mathscr K}}
\def\sL{{\mathscr L}}

\def\sP{{\mathscr P}}

\def\geq{\geqslant}
\def\leq{\leqslant}

\def\T{\mathord{{\rm Tr}}}

\allowdisplaybreaks

\begin{document}

\title{Functional law of large numbers and central limit theorem for slow-fast McKean-Vlasov equations}

\date{}

\author{Yun Li  and Longjie Xie}

\address{Yun Li:
       Institute of Systems Science,
       Academy of Mathematics and Systems Science, Chinese Academy of Sciences, and School of Mathematics Sciences, University of Chinese Academy of Sciences, Beijing 100149, P.R.China\\
       Email: liyun@amss.ac.cn
}

\address{Longjie Xie:
	School of Mathematics and Statistics, Jiangsu Normal University,
	Xuzhou, Jiangsu 221000, P.R.China\\
	Email: longjiexie@jsnu.edu.cn
}

\thanks{
	This work is supported by NNSF of China (No. 12090011, 12071186,  11931004).
}

\begin{abstract}
In this paper, we study the asymptotic behavior of  a fully-coupled slow-fast McKean-Vlasov stochastic system. Using the non-linear Poisson equation on Wasserstein space, we first establish the strong convergence in the averaging principle of the functional law of large numbers type. In particular, the diffusion coefficient of the slow process can depend on the distribution of  the  fast motion. Then we consider the stochastic fluctuations of the original system around its average, and prove that the normalized difference will converge weakly to a linear McKean-Vlasov Ornstein-Uhlenbeck type process, which can be viewed as a functional central limit theorem. Extra drift and diffusion coefficients involving the expectation are characterized explicitly. Furthermore, the optimal rates of the convergence are also obtained.

	\bigskip
	
	
\end{abstract}

\maketitle

\section{Introduction}

Consider the following slow-fast McKean-Vlasov stochastic differential equation (SDE for short)  in $\mR^{d_1}\times\mR^{d_2}$:
\begin{equation} \label{sde}
\left\{ \begin{aligned}
&\dif X^{\eps}_t =F(X^{\eps}_t,\cL_{X_t^\eps},Y^{\eps}_t,\cL_{Y^{\eps}_t})\dif t
+G(X^{\eps}_t,\cL_{X_t^\eps},\cL_{Y^{\eps}_t})\dif W^1_t,\qquad\qquad\quad X^{\eps}_0=\xi,\\
&\dif Y^{\eps}_t =\frac{1}{\eps}c(X^{\eps}_t,\cL_{X_t^\eps},Y^{\eps}_t,\cL_{Y^{\eps}_t})\dif t+\frac{1}{\eps^2}b(\cL_{X_t^\eps},Y^{\eps}_t,\cL_{Y^{\eps}_t})\dif t\\
&\qquad\quad+\frac{1}{\eps}\sigma_1(\cL_{X_t^\eps},Y^{\eps}_t,\cL_{Y^{\eps}_t})\dif W_t^1+\frac{1}{\eps}\sigma_2(\cL_{X_t^\eps},Y^{\eps}_t,\cL_{Y^{\eps}_t})\dif W_t^2,\quad\qquad   Y^{\eps}_0=\eta,
\end{aligned} \right.
\end{equation}
where $d_1, d_2\geq 1$,  $F: \mR^{d_1}\times\sP_2(\mR^{d_1})\times\mR^{d_2}\times\sP_2(\mR^{d_2})\to\mR^{d_1}$, $G: \mR^{d_1}\times\sP_2(\mR^{d_1})\times\sP_2(\mR^{d_2})\to\mR^{d_1}\otimes\mR^{d_1}$, $c: \mR^{d_1}\times\sP_2(\mR^{d_1})\times\mR^{d_2}\times\sP_2(\mR^{d_2})\to\mR^{d_2}$, $b: \sP_2(\mR^{d_1})\times\mR^{d_2}\times\sP_2(\mR^{d_2})\to\mR^{d_2}$, $\sigma_1: \sP_2(\mR^{d_1})\times\mR^{d_2}\times\sP_2(\mR^{d_2})\to\mR^{d_2}\otimes\mR^{d_1}$ and $\sigma_2: \sP_2(\mR^{d_1})\times\mR^{d_2}\times\sP_2(\mR^{d_2})\to\mR^{d_2}\otimes\mR^{d_2}$ are measurable functions, $W^1_t$, $W^2_t$ are $d_1$, $d_2$-dimensional independent standard Brownian motions both defined on some probability space $(\Omega,\sF,\mP)$, and $\xi$, $\eta$ are $d_1$, $d_2$-dimensional random variables, respectively. The small parameter $0<\eps\ll 1$ represents the separation of time scales between the slow component $X_t^\eps$ (which can be thought of as the mathematical model for
a phenomenon appearing at the natural time scale) and the fast motion $Y_t^\eps$ (which can be interpreted as the fast varying environment). Here and throughout, we denote by  $\cL_\varsigma$ the distribution of  a random
variable $\varsigma$, and
$\sP_2(\mR^{d})$ ($d\geq 1$) the space of all square integrable probability measures over $\mR^d$ equipped with the Wasserstein metric, i.e.,
$$
\cW_2(\mu_1,\mu_2):=\inf_{\pi\in\cP(\mu_1,\mu_2)}\left(\int_{\mR^d\times\mR^d}|x-y|^2\pi(\dif x,\dif y)\right)^{\frac{1}{2}},\quad \forall \mu_1,\mu_2\in \sP_2(\mR^{d}),
$$
where $\cP(\mu_1,\mu_2)$ is the class of measures on $\mR^d\times\mR^d$ with marginal $\mu_1$ and $\mu_2$.

\vspace{2mm}
The McKean-Vlasov SDE describes the limiting behavior of an individual particle involving within a system of particles interacting through their empirical measure, as the size of the population grows to infinity (the so-called propagation of chaos, see e.g. \cite{S}). A distinct feature is that the coefficients in  the equation depend not only on the solution process itself but also on its time marginal distributions. The pioneer work on such systems was indicated by Kac \cite{K} in kinetic theory and  McKean \cite{M} in the study of non-linear partial differential equations (PDEs for short). So far, the McKean-Vlasov SDEs have been investigated in various aspects such as well-posedness, connection with non-linear Fokker-Planck equations, large deviation  and numerical approximation, etc. We refer the readers to \cite{BH,BSY,SY,BR,CD,CF,CM,HW,MV,RZ} and the references therein.
Meanwhile, multi-scale models have wide range of applications  including climate weather interactions, intracellular biochemical reactions, geophysical fluid flows and stochastic volatility in finance, etc., and have been the central topic of study in science and engineering (see e.g. \cite{BYY,HL,PS,RX,RXY,S1,WR,X}). In particular,  multiple scales can leads to hysteresis loops in the bifurcation diagram and induce phase transitions of certain McKean-Vlasov equation as studied in \cite{DGP,GP}.

\vspace{2mm}Due to the widely separated  time scales and the cross interactions between the slow and fast
motions, the multi-scale McKean-Vlasov equations turn out to be more difficult to deal with. Hence, a simplified equation
which governs the evolution of the system over the long time scale is highly desirable. In this
direction, the theory of averaging principle provides a good approximation for the slow
component.
Existing averaging results for the multi-scale McKean-Vlasov SDEs can be found in \cite{BS1,BS2,BS3,HLL1,HLL2,RSX}. However, the coefficients of the systems considered in these works are not allowed to rely on the distribution of the fast motion. Recently, a system of weakly interacting diffusions in a two-scale potential relying on the faster empirical measure was considered in \cite{DGP}, the combined mean field and diffusive limits were investigated. The authors in \cite{LWX} considered the diffusion approximation of the multi-scale McKean-Vlasov SDEs by using a non-linear PDE as the corrector, where the coefficients  depend on the distributions of both the slow component
and the fast motion, yet only weak convergence is established therein.

\vspace{2mm}
In this paper, we shall first prove the strong convergence in the averaging principle for the system (\ref{sde}), see {\bf Theorem \ref{main1}} below. More precisely, we  show that as $\eps\to0$  the slow component $X_t^\eps$ will converge in $L^2(\Omega)$ to $\bar X_t$ which satisfies the following McKean-Vlasov equation:
\begin{align}\label{ave}
\dif \bar X_t&=\bar F(\bar X_t,\cL_{\bar X_t})\dif t+\bar G(\bar X_t,\cL_{\bar X_t})\dif W^1_t,\qquad  \bar X_0=\xi,
\end{align}
where the averaged coefficients are defined by
\begin{align}
\bar F(x,\mu)&:=\int_{\mR^{d_2}}F(x,\mu,y,\zeta^{\mu})\zeta^{\mu}(\dif y),\label{bF}\\
\bar G(x,\mu)&:=G(x,\mu,\zeta^{\mu}), \label{sigmaG}
\end{align}
and $\zeta^\mu(dy)$ is the unique invariant measure of the following parameterized McKean-Vlasov equation: for fixed $\mu\in\sP_2(\mR^{d_1})$,
\begin{align}\label{sde1}
\dif Y_t^{\mu,\eta}&=b(\mu,Y_t^{\mu,\eta},\cL_{Y^{\mu,\eta}_t})\dif t+\sigma_1(\mu,Y_t^{\mu,\eta},\cL_{Y^{\mu,\eta}_t})\dif \hat W_t^1\no\\
&\qquad\qquad\qquad\qquad\quad\,+\sigma_2(\mu,Y_t^{\mu,\eta},\cL_{Y^{\mu,\eta}_t})\dif \hat W_t^2,\quad Y_0^{\mu,\eta}=\eta,
\end{align}
where  $\hat W_t^1=\eps^{-1}W^1_{\eps^2t}$ and $\hat W_t^2=\eps^{-1}W^2_{\eps^2t}$ are two new independent Brownian motions. We point out that for the systems considered in \cite{BS1,BS2,BS3,HLL1,HLL2,RSX}, where the coefficients do not depend on the distribution of the fast motion, the equation (\ref{sde1}) then reduces to the classical It\^o's SDE (distribution-independent case), which is much easier to handle.  Moreover, it is interesting to note that we allow the diffusion coefficient $G$ in the slow process of system (\ref{sde}) to depend on the distribution of the fast variable, while it is well-known that in the theory of averaging principle of classical SDEs, the strong convergence   may not hold when the diffusion coefficient depends on the fast variable, see Remark \ref{br1} below.
This involves a delicate analysis on the convergence in the Wasserstein distance of the distribution of the fast process $Y_t^\eps$ to $\zeta^{\cL_{\bar X_t}}$. Note that in the definitions (\ref{bF}) and (\ref{sigmaG}) of the limit coefficients, we have freezed the $\nu$-measure variable of the coefficients $F$ and $G$  at the invariant measure $\zeta^\mu$.

\vspace{2mm}
Next, we proceed to study the small fluctuations of the slow component $X_t^\eps$  around its average $\bar X_t$, which is form of functional central limit type theorem. Namely,  we are interested in the asymptotic behavior of the normalized difference
$$
Z_t^\eps:=\frac{X_t^\eps-\bar X_t}{\eps}
$$
as $\eps\to0$.
We shall show that as $\eps\to0$, the deviation process $Z_t^\eps$ converges weakly to the  solution of a linear Ornstein-Uhlenbeck type McKean-Vlasov SDE, see equation (\ref{lim}) and {\bf Theorem \ref{main2}} below.
In particular,  the average effect of the drift term $c$ in the original system (\ref{sde}) will appear in the limit (even though it does not appear in the averaged equation (\ref{ave})). Meanwhile,  several extra drift and diffusion terms arise, which are explicitly characterized in terms of the  solution of an auxiliary non-linear Poisson equation  on the Wasserstein space. We provide two interesting particular cases to illustrate the result, see Remark \ref{br2} and  Example \ref{ex} below.

\vspace{2mm}
The rest of this paper is organized as follows. In Section 2, we state the main results.  In Section 3, we prepare some auxiliary results about the non-linear Poisson equation on the whole space and some a priori estimates. Section 4 is devoted to the proof of the strong convergence in the averaging principle. In Section 5, we establish a fluctuation lemma and then  give the proof of the functional central limit type theorem. Finally, an It\^o's formula is provided  in the Appendix for the sake of complicity.

\vspace{2mm}
\noindent{\bf Notations.}  To end this section, we introduce some notations. Throughout this paper, the letter $C$ with or without subscript denotes positive constant whose value may changes from line to line.
For brevity, we define
\begin{align}\label{op1}
\sL_1:=\sL_1(x,\mu,y,\nu):=\sum_{i=1}^{d_1}F_i(x,\mu,y,\nu)\frac{\p}{\p x_i}
+\frac{1}{2}\sum_{i,j=1}^{d_1}\big(GG^*(x,\mu,\nu)\big)_{ij}\frac{\p^2}{\p x_i\p x_j},
\end{align}
and
\begin{align}\label{op2}
\sL_2:=\sL_2(x,\mu,y,\nu):=\sum_{i=1}^{d_2}c_i(x,\mu,y,\nu)\frac{\p}{\p y_i}.
\end{align}
For a function $f(x,\mu,y,\nu)$ on $\mR^{d_1}\times\sP_2(\mR^{d_1})\times\mR^{d_2}\times\sP_2(\mR^{d_2})$, we say $f$ is Lipschitz continuous if
there exists a positive constant $L$ such that for every $x_1, x_2\in\mR^{d_1}$, $\mu_1,\mu_2\in\sP_2(\mR^{d_1})$, $y_1, y_2\in\mR^{d_2}$ and $\nu_1, \nu_2\in \sP_2(\mR^{d_2})$,
\begin{align*}
&|f(x_1,\mu_1,y_1,\nu_1)-f(x_2,\mu_2,y_2,\nu_2)|\\
&\leq L\big(|x_1-x_2|+|y_1-y_2|+\cW_2(\mu_1,\mu_2)+\cW_2(\nu_1,\nu_2)\big).
\end{align*}

\vspace{2mm}
Let us briefly recall the derivatives with respect to the measure variable introduced by Lions; see \cite[Section 6]{C} or \cite{Lions} for more details. The idea is to  consider  the canonical lift of a real-valued function $f: \sP_2(\mR^d)\to\mR$ into a function $\mF: L^2(\Omega)\ni X\mapsto f(\cL_X)\in\mR$. Using the Hilbert structure of the space $L^2(\Omega)$, the function $f$ is  said to be differentiable at $\mu\in \sP_2(\mR^d)$ if its canonical lift $\mF$ is
Fr\'echet differentiable at some point $X$ with $\cL_X=\mu$. By Riezs' representation theorem, the Fr\'echet derivative $D\mF(X)$, viewed as an element of $L^2(\Omega)$, can be given by a function $\p_\mu f(\mu)(\cdot): \mR^d\mapsto\mR^d$ such that
$$
D\mF(X)=\p_\mu f(\cL_X)(X).
$$
The function $\p_\mu f(\mu)(x)$ is then called the Lions derivative ($L$-derivative for short) of $f$ at $\mu$. Similarly, we can define the higher order derivatives of $f$ at $\mu$.

\vspace{2mm}

Let $d_1,d_2\geq 1$ and $k,\ell,m\in\mN=\{0,1,2,\cdots\}$. We introduce the following spaces of functions.

\begin{itemize}
  \item The space $C_b^{m,(k,\ell),2k,(k,k)}(\mR^{d_1}\times\sP_2(\mR^{d_1})\times\mR^{d_2}\times\sP_2(\mR^{d_2}))$. A function $f(x,\mu,y,\nu)$ is in $C_b^{m,(k,\ell),2k,(k,k)}(\mR^{d_1}\times\sP_2(\mR^{d_1})\times\mR^{d_2}\times\sP_2(\mR^{d_2}))$ if for any $(\mu,y,\nu)$, the mapping  $x\mapsto f(x,\mu,y,\nu)$ is in $C_b^{m}(\mR^{d_1})$, and for any $(x,y,\nu)$, the mapping  $\mu\mapsto f(x,\mu,y,\nu)$ is in $C_b^{(k,\ell)}(\sP_2(\mR^{d_1}))$, and for fixed $(x,\mu)$, the mapping $(y,\nu)\mapsto f(x,\mu,y,\nu)$ is in $C_b^{2k,(k,k)}(\mR^{d_2}\times\sP_2(\mR^{d_2}))$.

   \item The space $\mC_b^{(k,\ell),2k,(k,k)}(\sP_2(\mR^{d_1})\times\mR^{d_2}\times\sP_2(\mR^{d_2}))$. A function $f(\mu,y,\nu)$ is in $\mC_b^{(k,\ell),2k,(k,k)}(\sP_2(\mR^{d_1})\times\mR^{d_2}\times\sP_2(\mR^{d_2}))$ if $f\in C_b^{(k,\ell),2k,(k,k)}(\sP_2(\mR^{d_1})\times\mR^{d_2}\times\sP_2(\mR^{d_2}))$, and we can find a version of $\p^k_\mu f(\mu,y,\nu)(\tilde x_1,\cdots,\tilde x_k)$ such that the mapping $(y,\nu)\mapsto\p^\ell_{(\tilde x_1,\cdots,\tilde x_k)}\p^k_\mu f(\mu,y,\nu)(\tilde x_1,\cdots,\tilde x_k)$ is in $C_b^{2k,(k,k)}(\mR^{d_2}\times\sP_2(\mR^{d_2}))$.

  \item The space $\mC_b^{m,(k,\ell),2k,(k,k)}(\mR^{d_1}\times\sP_2(\mR^{d_1})\times\mR^{d_2}\times\sP_2(\mR^{d_2}))$. A function $f(x,\mu,y,\nu)$ is in $\mC_b^{m,(k,\ell),2k,(k,k)}(\mR^{d_1}\times\sP_2(\mR^{d_1})\times\mR^{d_2}\times\sP_2(\mR^{d_2}))$ if $f\in C_b^{m,(k,\ell),2k,(k,k)}(\mR^{d_1}\times\sP_2(\mR^{d_1})\times\mR^{d_2}\times\sP_2(\mR^{d_2}))$, and for every $x\in\mR^{d_1}$, the mapping $(\mu,y,\nu)\mapsto f(x,\mu,y,\nu)$ is in  $\mC_b^{(k,\ell),2k,(k,k)}
     (\sP_2(\mR^{d_1})\times\mR^{d_2}\times\sP_2(\mR^{d_2}))$, the mapping $(\mu,y,\nu)\mapsto\p_x^m f(x,\mu,y,\nu)$ is in $\mC_b^{(k,\ell),2k,(k,k)}(\sP_2(\mR^{d_1})\times\mR^{d_2}\times\sP_2(\mR^{d_2}))$.
\end{itemize}

\section{Statement of main results}

Let us first introduce some basic assumptions.  Throughout this paper, we assume the
following condition holds:

\vspace{2mm}
\noindent{\bf (H$^{\sigma,b}$):} there exist constants $c_2>c_1\geq 0$ such that for every $\mu\in\sP_2(\mR^{d_1})$, $y_1, y_2\in\mR^{d_2}$ and $\nu_1, \nu_2\in \sP_2(\mR^{d_2})$,
\begin{align*}
&3\big(\|\sigma_1(\mu,y_1,\nu_1)-\sigma_1(\mu,y_2,\nu_2)\|^2+\|\sigma_2(\mu,y_1,\nu_1)-\sigma_2(\mu,y_2,\nu_2)\|^2\big)\\
&+2\<b(\mu,y_1,\nu_1)-b(\mu,y_2,\nu_2),y_1-y_2\>\leq c_1\cW_2(\nu_1,\nu_2)^2-c_2|y_1-y_2|^2.
\end{align*}

Recall that for fixed $\mu\in\sP_2(\mR^{d_1})$, $Y_t^{\mu,\eta}$ satisfies the McKean-Vlasov equation (\ref{sde1}). It turns out that the distribution of $Y_t^{\mu,\eta}$ only depends on $\eta$ through its distribution $\cL_\eta=\nu$. Thus, for a given measure $\nu\in\sP_2(\mR^{d_2})$, we can define a (non-linear) semigroup $\{P_t^{\mu,*}\}_{t\geq0}$ on $\sP_2(\mR^{d_2})$ by letting
$$
P_{t}^{\mu,*}\nu:=\cL_{Y_{t}^{\mu,\eta}}\quad {\text {with}}\quad \cL_\eta=\nu.
$$
We say that a probability measure $\zeta^\mu$ is  an invariant measure of the McKean-Vlasov equation (\ref{sde1}) or the process $Y_t^{\mu,\eta}$ if
$$
P_t^{\mu,*}\zeta^\mu=\zeta^\mu,\quad \forall t\geq 0.
$$
Under assumption {\bf (H$^{\sigma,b}$)}, it is known that (see e.g. \cite[Theorem 3.1]{W1}) there exists a unique invariant measure $\zeta^\mu(\dif y)$ for the equation (\ref{sde1}). Moreover, there exist constants $C_0,\lambda_0>0$ such that for every $\nu\in\sP_2(\mR^{d_2})$,
\begin{align}\label{exp}
\cW_2(P_t^{\mu,*}\nu,\zeta^\mu)\leq C_0\,\e^{-\lambda_0t}\,\cW_2(\nu,\zeta^\mu).
\end{align}

\vspace{2mm}
The following is the first main result of this paper.

\bt\label{main1}
{\rm(Strong convergence)}. Let  {\bf (H$^{\sigma,b}$)} holds. Assume that $G$ and $c$ are Lipschitz continuous, $F\in C_b^{2,(1,1),2,(1,1)}(\mR^{d_1}\times\sP_2(\mR^{d_1})\times\mR^{d_2}\times\sP_2(\mR^{d_2}))$ with  $\p_xF(x,\mu,\cdot,\nu)\in C_b^1(\mR^{d_2})$ and $\sigma_1,\sigma_2,b\in C_b^{(1,1),2,(1,1)}( \sP_2(\mR^{d_1})\times\mR^{d_2}\times\sP_2(\mR^{d_2}))$. Then we have for any $T>0$,
\begin{align*}
\sup_{t\in[0,T]}\mE|X_t^\eps-\bar X_t|^2\leq C_T\,\eps^2,
\end{align*}
where $X_t^\eps$ and $\bar X_t$ satisfy the McKean-Vlasov equation (\ref{sde}) and (\ref{ave}), respectively, and $C_T>0$ is a constant independent of $\eps$.
\et

\br\label{br1}
(i) We shall show that the averaged coefficients $\bar F$ defined in (\ref{bF})  and $\bar G$ defined in (\ref{sigmaG}) are Lipschitz continuous with respect to $(x,\mu)$ (see Lemma \ref{Y3} below). Thus, there exists a unique strong solution $\bar X_t$ to the averaged McKean-Vlasov equation (\ref{ave}).

\vspace{2mm}
(ii) In the theory of the averaging principle of classical SDEs (i.e., when the coefficients in system (\ref{sde}) do not depend on the distribution of the solution), counter example is known which shows that the strong convergence does not hold when the diffusion coefficient $G$ in the slow variable depends on the fast motion $Y_t^\eps$, see e.g. \cite{L}. Here, we show that the diffusion coefficient $G$ can depend on the distribution $\cL_{Y_t^\eps}$ of the fast motion. This involves the convergence in Wasserstein distance of the distribution of $Y_t^\eps$.
\er

Next, we proceed to identify the limit of the normalized difference
\begin{align}\label{z}
Z_t^\eps:=\frac{X_t^\eps-\bar X_t}{\eps}
\end{align}
as $\eps\to0$. For this,
we assume
\begin{align}\label{G}
G(x,\mu,\nu)\equiv G(x,\mu).
\end{align}
As a result, we have
\begin{align*}
Z_t^\eps&=\frac{1}{\eps}\int_0^t\Big[F(X^{\eps}_s,\cL_{X_s^\eps},Y^{\eps}_s,\cL_{Y^{\eps}_s})-\bar F(\bar X_s,\cL_{\bar X_s})\Big]\dif s\\
&\quad+\frac{1}{\eps}\int_0^t\Big[G(X^{\eps}_s,\cL_{X_s^\eps})-G(\bar X_s,\cL_{\bar X_s})\Big]\dif W^1_s,
\end{align*}
where $\bar F(x,\mu)$ is given by (\ref{bF}). To introduce the limit $\bar Z_t$ of $Z_t^\eps$, we need to consider the following Poisson equation on $\mR^{d_2}\times\sP_2(\mR^{d_2})$:
\begin{align}\label{pss2}
\sL_0(\mu,y,\nu)\Phi(x,\mu,y,\nu)=-[F(x,\mu,y,\nu)-\bar F(x,\mu)]=:-\delta F(x,\mu,y,\nu),
\end{align}
where $(x,\mu)\in\mR^{d_1}\times\sP_2(\mR^{d_1})$ are regarded as parameters, and for a test function $\varphi(y,\nu)$, the operator $\sL_0$ is defined by
\begin{align}\label{op}
&\sL_0\varphi(y,\nu):=\sL_0(\mu,y,\nu)\varphi(y,\nu)\no\\
&:=\frac{1}{2}\T\big(a(\mu,y,\nu)\cdot\p^2_y\varphi(y,\nu)\big)+b(\mu,y,\nu)\cdot\p_y\varphi(y,\nu)\no\\
&\quad+\int_{\mR^{d_2}}\Big[\frac{1}{2}\T\big(a(\mu,\tilde y,\nu)\cdot\p_{\tilde y}[\p_\nu\varphi(y,\nu)(\tilde y)]\big)+b(\mu,\tilde y,\nu)\cdot\p_\nu\varphi(y,\nu)(\tilde y)\Big]\nu(\dif \tilde y),
\end{align}
with $a(\mu,y,\nu)=[\sigma_1\sigma_1^*+\sigma_2\sigma_2^*](\mu,y,\nu)$.
In fact, the operator $\sL_0$ can be viewed as the infinitesimal generator of the parameterized  McKean-Vlasov SDE (\ref{sde1}). Note that the equation (\ref{pss2}) is totally non-linear due to the existence of the integral part with respect to the measure $\nu$ in (\ref{op}).
According to Theorem \ref{pfs} below, there exists a unique solution $\Phi(x,\mu,y,\nu)$ to the equation (\ref{pss2}). Let us define
\begin{align}
\overline{c\cdot\p_y\Phi}(x,\mu)&:=\int_{\mR^{d_2}}c(x,\mu,y,\zeta^{\mu})\cdot\p_y\Phi(x,\mu,y,\zeta^\mu)\zeta^\mu(\dif y),\label{bc1}\\
\overline{\overline{c\cdot\p_\nu\Phi}}(x,\mu)(\tilde x)&:=\int_{\mR^{d_2}}\!\!\int_{\mR^{d_2}}\!\!c(\tilde x,\mu,\tilde y,\zeta^{\mu})\cdot\p_\nu\Phi(x,\mu,y,\zeta^\mu)(\tilde y)\zeta^\mu(\dif\tilde y)\zeta^\mu(\dif y),\label{bc2}\\
\overline{\sigma_1^*\cdot\p_x\p_y\Phi}(x,\mu)&:=\int_{\mR^{d_2}}\sigma_1^*(\mu,y,\zeta^\mu)
\cdot\p_x\p_y\Phi(x,\mu,y,\zeta^\mu)\zeta^\mu(\dif y),\label{bG}\\
\overline{\p_y\Phi\cdot\sigma_1}(x,\mu)&:=\int_{\mR^{d_2}}\p_y\Phi(x,\mu,y,\zeta^\mu)\cdot\sigma_1(\mu,y,\zeta^\mu)\zeta^\mu(\dif y),\label{sigma1}\\
\overline{(\p_y\Phi\cdot\sigma_1)(\p_y\Phi\cdot\sigma_1)^*}(x,\mu)
&:=\int_{\mR^{d_2}}
\Big(\p_y\Phi(x,\mu,y,\zeta^\mu)\cdot\sigma_1(\mu,y,\zeta^\mu)\Big)\no\\
&\qquad\qquad\Big(\p_y\Phi(x,\mu,y,\zeta^\mu)\cdot\sigma_1(\mu,y,\zeta^\mu)\Big)^*\zeta^\mu(\dif y),\label{sigma11}
\end{align}
and
\begin{align}
\overline{(\p_y\Phi\cdot\sigma_2)(\p_y\Phi\cdot\sigma_2)^*}(x,\mu)
&:=\int_{\mR^{d_2}}
\Big(\p_y\Phi(x,\mu,y,\zeta^\mu)\cdot\sigma_2(\mu,y,\zeta^\mu)\Big)\no\\
&\qquad\qquad\Big(\p_y\Phi(x,\mu,y,\zeta^\mu)\cdot\sigma_2(\mu,y,\zeta^\mu)\Big)^*\zeta^\mu(\dif y).\label{sigma2}
\end{align}
Then the limit $\bar Z_t$ for the deviation process $Z_t^\eps$ turns out to satisfy the following McKean-Vlasov SDE:
\begin{align}\label{lim}
\dif \bar Z_t&=\p_x\bar F(\bar X_t,\cL_{\bar X_t})\bar Z_t\dif t+\tilde\mE\big[\p_\mu\bar F(\bar X_t,\cL_{\bar X_t})(\tilde{\bar X}_t)\tilde{\bar Z}_t\big]\dif t\no\\
&\quad+\overline{c\cdot\p_y\Phi}(\bar X_t,\cL_{\bar X_t})\dif t+\tilde\mE\left[\overline{\overline{c\cdot\p_\nu\Phi}}(\bar X_t,\cL_{\bar X_t})(\tilde{\bar X}_t)\right]\dif t\no\\
&\quad+\overline{\sigma_1^*\cdot\p_x\p_y\Phi}(\bar X_t,\cL_{\bar X_t})\cdot G(\bar X_t,\cL_{\bar X_t})\dif t\no\\
&\quad+\p_xG(\bar X_t,\cL_{\bar X_t})\bar Z_t\dif W_t^1+\tilde\mE\big[\p_\mu G(\bar X_t,\cL_{\bar X_t})(\tilde{\bar X}_t)\tilde{\bar Z}_t\big]\dif W_t^1\no\\
&\quad+\overline{\p_y\Phi\cdot\sigma_1}(\bar X_t,\cL_{\bar X_t})\dif W_t^1+\sqrt{\Sigma(\bar X_t,\cL_{\bar X_t})}\dif \tilde W_t,\qquad \bar Z_0=0,
\end{align}
where $\tilde W_t$ is another Brownian motion independent of $W_t^1$, $\bar X_t$ is the unique strong solution of the averaged equation (\ref{ave}), the process $(\tilde{\bar X}_t,\tilde{\bar Z}_t)$ is a copy of $(\bar X_t,\bar Z_t)$ defined on a copy $(\tilde\Omega,\tilde\sF, \tilde{\mP})$ of the original probability space  $(\Omega,\sF, \mP)$, and $\tilde\mE$ is the expectation taken with respect to $\tilde\mP$.  The diffusion coefficient  $\Sigma(x,\mu)$ is defined by
\begin{align}\label{sigma}
\Sigma(x,\mu)&=\overline{(\p_y\Phi\cdot\sigma_1)(\p_y\Phi\cdot\sigma_1)^*}(x,\mu)
-\overline{\p_y\Phi\cdot\sigma_1}\cdot\overline{(\p_y\Phi\cdot\sigma_1)^*}(x,\mu)\no\\
&\quad+\overline{(\p_y\Phi\cdot\sigma_2)(\p_y\Phi\cdot\sigma_2)^*}(x,\mu).
\end{align}
 Note that the matrix $\Sigma(x,\mu)$ is always positive semi-definite.

\vspace{2mm}
The following is the second main result of this paper.
\bt\label{main2}
{\rm(Central limit theorem)}. Let assumptions  {\bf (H$^{\sigma,b}$)} and (\ref{G}) hold. Assume that $G\in\big(C_b^{4,(1,3)}\cap C_b^{4,(2,2)}\cap C_b^{4,(3,1)}\big)(\mR^{d_1}\times\sP_2(\mR^{d_1}))$, $c\in  \big(C_b^{3,(2,2),4,(2,2)}\cap\mC_b^{3,(1,2),2,(1,1)}\big)(\mR^{d_1}\times\sP_2(\mR^{d_1})\times\mR^{d_2}\times\sP_2(\mR^{d_2}))$, $\sigma_1,\sigma_2,b\in \big(C_b^{(3,1),6,(3,3)}\cap\mC_b^{(1,3),4,(2,2)}\cap\mC_b^{(2,2),2,(1,1)}\big)(\sP_2(\mR^{d_1})\times\mR^{d_2}\times\sP_2(\mR^{d_2}))$ and $F\in\big(C_b^{4,(3,1),6,(3,3)}\cap\mC_b^{4,(1,3),4,(2,2)}\cap\mC_b^{4,(2,2),2,(1,1)}\big)(\mR^{d_1}\times\sP_2(\mR^{d_1})\times\mR^{d_2}\times\sP_2(\mR^{d_2}))$. Then for any $T>0$ and $\varphi\in \big(C_b^{(1,3)}\cap C_b^{(2,2)}\cap C_b^{(3,1)}\big)(\sP_2(\mR^{d_1}))$, we have
\begin{align*}
\sup_{t\in[0,T]}\big|\varphi(\cL_{Z_t^\eps})-\varphi(\cL_{\bar Z_t})\big|\leq C_T\,\eps,
\end{align*}
where $\bar Z_t$ satisfies the McKean-Vlasov equation (\ref{lim}), and $C_T>0$ is a constant independent of $\eps$. In particular, we have for every $\phi\in C_b^{4}(\mR^{d_1})$,
\begin{align*}
\sup_{t\in[0,T]}\big|\mE\phi(Z_t^\eps)-\mE\phi(\bar Z_t)\big|\leq C_T\,\eps.
\end{align*}
\et

\br\label{br2}
(i) Note that the average effect of the term $c$ in the system (\ref{sde}) does not appear in the averaging principle of the law of large number (i.e., it does not appear in equation (\ref{ave})). However, in the central limit theorem it arises in the  equation (\ref{lim}). The expectation term involving $L$-derivative in the measure argument of the solution $\Phi$ in (\ref{lim})  is  due to the effect of the dependence on the fast distribution in the coefficients.

\vspace{1mm}
(ii) The terms involving $\sigma_1$ in (\ref{lim}) seem to be  new even for classical SDEs, which is due to the effect of the common noise.
\er


\bx\label{ex}
When $G\equiv\mI_{d_1}$ (the identity matrix) in system (\ref{sde}), then the corresponding averaged equation (\ref{ave}) becomes
$$
\dif \bar X_t=\bar F(\bar X_t,\cL_{\bar X_t})\dif t+\dif W_t^1.
$$
As a result, we have
\begin{align*}
Z_t^\eps:&=\frac{X_t^\eps-\bar X_t}{\eps}=\frac{1}{\eps}\int_0^t\Big[F(X_s^\eps,\cL_{X_s^\eps},Y_s^\eps,\cL_{Y_s^\eps})-\bar F(\bar X_s,\cL_{\bar X_s})\Big]\dif s\\
&=\frac{1}{\eps}\int_0^t\Big[\bar F(X_s^\eps,\cL_{X_s^\eps})-\bar F(\bar X_s,\cL_{\bar X_s})\Big]\dif s\\
&\quad+\frac{1}{\eps}\int_0^t\Big[F(X_s^\eps,\cL_{X_s^\eps},Y_s^\eps,\cL_{Y_s^\eps})-\bar F(X_s^\eps,\cL_{X_s^\eps})\Big]\dif s=:\cI_1(\eps)+\cI_2(\eps).
\end{align*}
Let $\bar Z_t$ be the limit of $Z_t^\eps$. Then, at least formally,  we have by the mean value theorem that
$$
\cI_1(\eps)\to \int_0^t\Big[\p_x\bar F(\bar X_s,\cL_{\bar X_s})\bar Z_s+\tilde\mE\big[\p_\mu\bar F(\bar X_s,\cL_{\bar X_s})(\tilde{\bar X}_s)\tilde{\bar Z}_s\big]\Big]\dif s\quad\text{as}\quad\eps\to0.
$$
The limit for the second term $\cI_2(\eps)$ is far from being obvious. We provide two cases to illustrate the result.

(i) When $\sigma_2=\mI_{d_2}$ and $\sigma_1=c\equiv0$ in system (\ref{sde}), then according to Theorem \ref{main2}, we have
$$
\cI_2(\eps)\to\int_0^t\sqrt{\Sigma(\bar X_s,\cL_{\bar X_s})}\dif \tilde W_s,
$$
where $\tilde W_t$ is a new Brownian motion independent of $W_t^1$, and
$$
\Sigma(x,\mu):=\overline{(\p_y\Phi)(\p_y\Phi)^*}(x,\mu).
$$
Thus the limit $\bar Z_t$ satisfies the linear McKean-Vlasov equation
\begin{align*}
\dif \bar Z_t&=\p_x\bar F(\bar X_t,\cL_{\bar X_t})\bar Z_t\dif t+\tilde\mE\big[\p_\mu\bar F(\bar X_t,\cL_{\bar X_t})(\tilde{\bar X}_t)\tilde{\bar Z}_t\big]\dif t\\
&\quad+\sqrt{\Sigma(\bar X_t,\cL_{\bar X_t})}\dif \tilde W_t.
\end{align*}

(ii) When $\sigma_1=\mI_{d_1}$ (with common noise) and $\sigma_2=c\equiv0$, then according to Theorem \ref{main2}, we have
\begin{align*}
\cI_2(\eps)\to\int_0^t\overline{\p_x\p_y\Phi}(\bar X_s,\cL_{\bar X_s})\dif s+\int_0^t\overline{\p_y\Phi}(\bar X_s,\cL_{\bar X_s})\dif W^1_s+\int_0^t\sqrt{\Sigma(\bar X_s,\cL_{\bar X_s})}\dif \tilde W_s,
\end{align*}
where $\tilde W_t$ is a new Brownian motion independent of $W_t^1$, and
$$
\Sigma(x,\mu):=\overline{(\p_y\Phi)(\p_y\Phi)^*}(x,\mu)
-\overline{\p_y\Phi}\cdot\overline{(\p_y\Phi)^*}(x,\mu).
$$
Thus the limit $\bar Z_t$ satisfies
\begin{align*}
\dif \bar Z_t&=\p_x\bar F(\bar X_t,\cL_{\bar X_t})\bar Z_t\dif t+\tilde\mE\big[\p_\mu\bar F(\bar X_t,\cL_{\bar X_t})(\tilde{\bar X}_t)\tilde{\bar Z}_t\big]\dif t\\
&\quad+\overline{\p_x\p_y\Phi}(\bar X_t,\cL_{\bar X_t})\dif t+\overline{\p_y\Phi}(\bar X_t,\cL_{\bar X_t})\dif W^1_t+\sqrt{\Sigma(\bar X_t,\cL_{\bar X_t})}\dif \tilde W_t.
\end{align*}

\ex

\section{Poisson equation and auxiliary estimates}

In this section, we first recall some results about the Poisson equation on the Wasserstein space.  Then, we collect some a priori estimates that we shall use to prove our main results.

Consider the following Poisson equation on the whole space $\mR^{d_2}\times\sP_2(\mR^{d_2})$:
\begin{align}\label{pof}
\sL_0(\mu,y,\nu)U(x,\mu,y,\nu)=-f(x,\mu,y,\nu),
\end{align}
where $(x,\mu)\in\mR^{d_1}\times\sP_2(\mR^{d_1})$  are parameters, and the operator $\sL_0$ is defined by (\ref{op}).
In order to ensure the well-posedness of the equation (\ref{pof}), we need to assume that $f$ satisfies the following centering condition:
\begin{align}\label{cenf}
\int_{\mR^{d_2}}f(x,\mu,y,\zeta^{\mu})\zeta^{\mu}(\dif y)=0,\quad\forall (x,\mu)\in\mR^{d_1}\times\sP_2(\mR^{d_1}),
\end{align}
where $\zeta^\mu(dy)$ is the unique invariant measure of the frozen McKean-Vlasov equation (\ref{sde1}). Furthermore, we need to consider the following  de-coupled equation associated with (\ref{sde1}):
\begin{align}\label{sde2}
Y_t^{\mu,y,\nu}=y+\int_0^tb\big(\mu,Y_s^{\mu,y,\nu},\cL_{Y_s^{\mu,\eta}}\big)\dif s&+\int_0^t\sigma_1\big(\mu,Y_s^{\mu,y,\nu},\cL_{Y_s^{\mu,\eta}}\big)\dif  W^1_s\no\\
&+\int_0^t\sigma_2\big(\mu,Y_s^{\mu,y,\nu},\cL_{Y_s^{\mu,\eta}}\big)\dif  W^2_s,
\end{align}
with $\cL_\eta=\nu$.
The results below were proved in \cite[Theorems 2.3, 2.4]{LWX}, which will be used frequently in the sequel.

\bt\label{pfs}
Let  {\bf (H$^{\sigma,b}$)} hold,  $j, k, m, n\in\mN$, and the function $f$ satisfy the centering condition  (\ref{cenf}).

\vspace{2mm}
(i) Assume that for every $(x,\mu)\in\mR^{d_1}\times\sP_2(\mR^{d_1})$, $a(\mu,\cdot,\cdot),b(\mu,\cdot,\cdot),f(x,\mu,\cdot,\cdot)\in C_b^{2m,(m,m)}(\mR^{d_2}\times\sP_2(\mR^{d_2}))$. Then there exists a unique solution $U(x,\mu,\cdot,\cdot)\in C_b^{2m,(m,m)}(\mR^{d_2}\times\sP_2(\mR^{d_2}))$ to the equation (\ref{pof}), which also satisfies the centering condition (\ref{cenf}) and is given by
\begin{align}\label{pfu}
U(x,\mu,y,\nu)&=\mE\left(\int_0^\infty f\big(x,\mu,Y_t^{\mu,y,\nu},\cL_{Y_t^{\mu,\eta}}\big)\dif t\right),
\end{align}
where $Y_t^{\mu,\eta}$ and $Y_t^{\mu,y,\nu}$ satisfy equations (\ref{sde1}) and (\ref{sde2}) with $\cL_\eta=\nu$, respectively.

\vspace{2mm}
(ii) Assume that $a, b\in \big(C_b^{(m,k),2m,(m,m)}\cap \mC_b^{(n,k),2(m-n),(m-n,m-n)}\big)(\sP_2(\mR^{d_1})\times\mR^{d_2}\times\sP_2(\mR^{d_2}))$ and  $f\in \big(C_b^{j,(m,k),2m,(m,m)}\cap \mC_b^{j,(n,k),2(m-n),(m-n,m-n)}\big)(\mR^{d_1}\times\sP_2(\mR^{d_1})\times\mR^{d_2}\times\sP_2(\mR^{d_2}))$ with $0\leq n<m$. Then we have
\begin{align*}
U\in \big(C_b^{j,(m,k),2m,(m,m)}\cap \mC_b^{j,(n,k),2(m-n),(m-n,m-n)}\big)(\mR^{d_1}\times\sP_2(\mR^{d_1})\times\mR^{d_2}\times\sP_2(\mR^{d_2})).
\end{align*}
\et

Given a function $h(x,\mu,y,\nu)$, we denote by $\bar h(x,\mu)$ its average with respect to the invariant measure $\zeta^\mu(dy)$, i.e.,
\begin{align*}
\bar h(x,\mu):=\int_{\mR^{d_2}}h(x,\mu,y,\zeta^{\mu})\zeta^{\mu}(\dif y).
\end{align*}
As a direct application of Theorem \ref{pfs}, we have the following result which illustrates the regularity of an averaged function.

\bl\label{reh}
Assume that  {\bf (H$^{\sigma,b}$)} holds and  $\ell, m, k\in\mN$. If for every $1\leq n<m$, $a, b\in (C_b^{(m,k),2m,(m,m)}\cap \mC_b^{(n,k),2(m-n),(m-n,m-n)})(\sP_2(\mR^{d_1})\times\mR^{d_2}\times\sP_2(\mR^{d_2}))$ and  $h\in (C_b^{\ell,(m,k),2m,(m,m)}\cap \mC_b^{\ell,(n,k),2(m-n),(m-n,m-n)})(\mR^{d_1}\times\sP_2(\mR^{d_1})\times\mR^{d_2}\times\sP_2(\mR^{d_2}))$. Then we have
$\bar h\in C_b^{\ell,(m,k)}(\mR^{d_1}\times\sP_2(\mR^{d_1}))$. In particular,

\vspace{2mm}
\item (i) under the assumptions in Theorem \ref{main1}, we have $\bar F\in C_b^{2,(1,1)}$;

\vspace{2mm}
\item (ii) under the assumptions in Theorem \ref{main2}, we have
$\bar F\in C_b^{4,(1,3)}\cap C_b^{4,(2,2)}\cap C_b^{4,(3,1)}$,
and
$$\overline{c\cdot\p_y\Phi},~\overline{\sigma_1^*\cdot\p_x\p_y\Phi},~ \overline{\p_y\Phi\cdot\sigma_1},~ \overline{\delta F\cdot\Phi^*}\in C_b^{3,(2,2)},~~~\overline{\overline{c\cdot\p_\nu\Phi}}\in C_b^{3,(2,2),3},
$$
where the above functions are defined by (\ref{bc1})-(\ref{sigma2}).
\el
\begin{proof}
The conclusion that $\bar h\in C_b^{\ell,(m,k)}(\mR^{d_1}\times\sP_2(\mR^{d_1}))$ was proved in \cite[Corollary 2.5]{LWX}. Then under the assumptions in Theorem \ref{main1}, we take $\ell=2$ and $m=k=1$ in the above, the assertion that $\bar F\in C_b^{2,(1,1)}(\mR^{d_1}\times\sP_2(\mR^{d_1}))$ follows directly.   Similarly, under the assumptions in Theorem \ref{main2}, we deduce that $\bar F\in C_b^{4,(1,3)}\cap C_b^{4,(2,2)}\cap C_b^{4,(3,1)}$. Recall that $\Phi$ solves (\ref{pss2}). By Theorem \ref{pfs}, we get $\Phi\in C_b^{4,(3,1),6,(3,3)}\cap \mC_b^{4,(1,3),4,(2,2)}\cap \mC_b^{4,(2,2),2,(1,1)}$. This together with the conditions on $c$ and $\sigma_1$ implies that $\overline{c\cdot\p_y\Phi}$, $\overline{\sigma_1^*\cdot\p_x\p_y\Phi}, \overline{\p_y\Phi\cdot\sigma_1}, \overline{\delta F\cdot\Phi^*}\in C_b^{3,(2,2)}$ and $\overline{\overline{c\cdot\p_\nu\Phi}}\in C_b^{3,(2,2),3}$.
\end{proof}

Let $(X_t^\eps,Y_t^\eps)$ satisfy the McKean-Vlasov equation (\ref{sde}). By using the similar arguments as \cite[Lemma 3.1]{RSX}, we have the following  moment estimates of the process $(X_t^\eps,Y_t^\eps)$, the details of the proof are omitted.

\bl\label{Y1}
Let {\bf (H$^{\sigma,b}$)} hold. Assume that $F,G$ and $c$ are Lipschitz continuous. Then for any $T>0$, there exists a positive constant $C_T$ such that
\begin{align*}
\sup_{0<\eps\ll1}\mE\Big[\sup_{t\in[0,T]}|X^\eps_t|^4\Big]\leq C_T\big(1+\mE|\xi|^4+\mE|\eta|^4\big),
\end{align*}
and
\begin{align*}
\sup_{0<\eps\ll1}\sup_{t\in[0,T]}\mE|Y^\eps_t|^4\leq C_T\big(1+\mE|\xi|^4+\mE|\eta|^4\big).
\end{align*}
\el

Recall that $Y_t^{\mu,\eta}$ is the unique strong solution of the equation (\ref{sde1}). We have:
\bl\label{Y2}
Assume that {\bf (H$^{\sigma,b}$)} holds. Then we have for any $t>0$, $\mu,\mu_1,\mu_2\in\sP_2(\mR^{d_1})$ and $\eta\in L^2(\Omega)$,
\begin{align*}
\mE|Y^{\mu_1,\eta}_t-Y^{\mu_2,\eta}_t|^2\leq C_0\cW_2(\mu_1,\mu_2)^2,
\end{align*}
and
\begin{align*}
\mE|Y^{\mu,\eta}_t|^2\leq C_0\big(e^{-(c_2-c_1)t}\mE|\eta|^2+\cW_2(\mu,\delta_0)^2\big),
\end{align*}
where $C_0>0$ is a constant independent of $t$.
\el
\begin{proof}
Using  It\^{o}'s formula and by {\bf (H$^{\sigma,b}$)}, we have that
\begin{align*}
\dif \mE|Y^{\mu_1,\eta}_t-Y^{\mu_2,\eta}_t|^2&=
\mE\big[2\langle Y^{\mu_1,\eta}_t-Y^{\mu_2,\eta}_t, b(\mu_1,Y^{\mu_1,\eta}_t,\cL_{Y^{\mu_1,\eta}_t})-b(\mu_2,Y^{\mu_2,\eta}_t,\cL_{Y^{\mu_2,\eta}_t}) \rangle\\
&\quad+\|\sigma_1(\mu_1,Y^{\mu_1,\eta}_t,\cL_{Y^{\mu_1,\eta}_t})-\sigma_1(\mu_2,Y^{\mu_2,\eta}_t,\cL_{Y^{\mu_2,\eta}_t})\|^2\\
&\quad+\|\sigma_2(\mu_1,Y^{\mu_1,\eta}_t,\cL_{Y^{\mu_1,\eta}_t})-\sigma_2(\mu_2,Y^{\mu_2,\eta}_t,\cL_{Y^{\mu_2,\eta}_t})\|^2\big]\dif t\\
&\leq -(c_2-c_1)\mE|Y^{\mu_1,\eta}_t-Y^{\mu_2,\eta}_t|^2\dif t+C_0\cW_2(\mu_1,\mu_2)^2\dif t,
\end{align*}
which together with the comparison theorem implies
\begin{align*}
\mE|Y^{\mu_1,\eta}_t-Y^{\mu_2,\eta}_t|^{2}\leq C_0\,\cW_2(\mu_1,\mu_2)^2.
\end{align*}

In view of {\bf (H$^{\sigma,b}$)}, for every $\mu\in\sP_2(\mR^{d_1})$, $y\in\mR^{d_2}$ and $\nu\in\sP_2(\mR^{d_2})$, there exists $C_0>0$ such that
\begin{align*}
&2\langle y, b(\mu,y,\nu)\rangle+3\|\sigma_1(\mu,y,\nu)\|^2+3\|\sigma_2(\mu,y,\nu)\|^2\\
&\leq c_1\cW_2(\nu,\delta_{0})^2-c_2|y|^2+C_0\cW_2(\mu,\delta_{0})^2.
\end{align*}
In the same way we get
\begin{align*}
\dif \mE|Y^{\mu,\eta}_t|^2&=
\mE\big[2\langle Y^{\mu,\eta}_t, b(\mu,Y^{\mu,\eta}_t,\cL_{Y^{\mu,\eta}_t}) \rangle\\
&\quad+\|\sigma_1(\mu,Y^{\mu,\eta}_t,\cL_{Y^{\mu,\eta}_t})\|^2+\|\sigma_2(\mu,Y^{\mu,\eta}_t,\cL_{Y^{\mu,\eta}_t})\|^2\big]\dif t\\
&\leq -(c_2-c_1)\mE|Y^{\mu,\eta}_t|^2\dif t+C_0\cW_2(\mu,\delta_0)^2\dif t,
\end{align*}
which in turn yields that
\begin{align*}
\mE|Y^{\mu,\eta}_t|^{2}\leq C_0\,e^{-(c_2-c_1)t}\mE|\eta|^2+C_0\,\cW_2(\mu,\delta_0)^2.
\end{align*}
Thus the desired results are proved.
\end{proof}

\bl
Assume that {\bf (H$^{\sigma,b}$)} holds. Then there exists a positive constant $C_0$ such that for any $\mu\in\sP_2(\mR^{d_1})$,
\begin{align}
&\cW_2(\zeta^\mu,\delta_0)^2\leq C_0\,\cW_2(\mu,\delta_0)^2.\label{zeta1}
\end{align}
\el

\begin{proof}
Using (\ref{exp}) and Lemma \ref{Y2}, we have
\begin{align*}
\cW_2(\zeta^\mu,\delta_0)^2&\leq2\cW_2(\zeta^\mu,P_t^{\mu,*}\delta_0)^2+2\cW_2(P_t^{\mu,*}\delta_0,\delta_0)^2\\
&\leq C_0\,e^{-2\lambda_0t}\cW_2(\zeta^\mu,\delta_0)^2+C_0\,\mE|Y^{\mu,0}_t|^2\\
&\leq C_0\,e^{-2\lambda_0t}\cW_2(\zeta^\mu,\delta_0)^2+C_0\,\cW_2(\mu,\delta_0)^2,
\end{align*}
where $C_0>0$ is a constant independent of $t$. Taking the limit $t\to\infty$, the desired conclusion follows.
\end{proof}

\section{Strong convergence in the averaging principle}

Using the technique of Poisson equation, we shall first establish a strong fluctuation estimate in Subsection 4.1. Then we prove the strong convergence of the slow-fast system (\ref{sde}) to the averaged system (\ref{ave}) in Subsection 4.2. The optimal rate of convergence follows as a byproduct.

\subsection{Strong Fluctuation estimate}

Given a function $f(x,\mu,y,\nu)$ on $\mR^{d_1}\times\sP_2(\mR^{d_1})\times\mR^{d_2}\times\sP_2(\mR^{d_2})$,
the following result gives an estimate for the fluctuations of the process $f(X_s^\eps,\cL_{X_s^\eps},Y_s^\eps,\cL_{Y_s^\eps})$ over the time interval $[0,t]$.

\bl\label{flu}
Let {\bf (H$^{\sigma,b}$)} hold. Assume that $F,G$ and $c$ are Lipschitz continuous and $\sigma_1,\sigma_2,b\in C_b^{(1,1),2,(1,1)}(\sP_2(\mR^{d_1})\times\mR^{d_2}\times\sP_2(\mR^{d_2}))$. Then for every $f\in C_b^{2,(1,1),2,(1,1)}(\mR^{d_1}\times\sP_2(\mR^{d_1})\times\mR^{d_2}\times\sP_2(\mR^{d_2}))$ satisfying (\ref{cenf}) and $\p_xf(x,\mu,\cdot,\nu)
\in C_b^1(\mR^{d_2})$, we have
\begin{align*}
\mE\left|\int_0^tf(X_s^\eps,\cL_{X_s^\eps},Y_s^\eps,\cL_{Y_s^\eps})\dif s\right|^2\leq C_0\,\eps^2,
\end{align*}
where $C_0>0$ is a constant independent of $\eps$.
\el
\begin{proof}
Since $f$ satisfies (\ref{cenf}), by the assumptions on the coefficients and Theorem \ref{pfs}, there exists a unique
solution $\psi\in C_b^{2,(1,1),2,(1,1)}(\mR^{d_1}\times\sP_2(\mR^{d_1})\times\mR^{d_2}\times\sP_2(\mR^{d_2}))$ to the following Poisson equation:
\begin{align}\label{pss}
\sL_0(\mu,y,\nu)\psi(x,\mu,y,\nu)=-f(x,\mu,y,\nu),
\end{align}
where $(x,\mu)\in\mR^{d_1}\times\sP_2(\mR^{d_1})$ are regarded as parameters, and the operator $\sL_0$ is defined by (\ref{op}). Moreover, we have $\p_x\psi(x,\mu,\cdot,\nu)
\in C_b^1(\mR^{d_2})$.
Using It\^o's formula (see Lemma \ref{ito} below or \cite[Proposition 2.1]{CF}), we deduce that
\begin{align}
&\psi(X_t^\eps,\cL_{X_t^\eps},Y_t^\eps,\cL_{Y_t^\eps})\no\\
&=\psi(\xi,\cL_\xi,\eta,\cL_\eta) +\int_0^t\Big[\sL_1(X_s^\eps,\cL_{X_s^\eps},Y_s^\eps,\cL_{Y_s^\eps})\psi(X_s^\eps,\cL_{X_s^\eps},Y_s^\eps,\cL_{Y_s^\eps})\no\\
&\quad+\frac{1}{\eps}\sL_2(X_s^\eps,\cL_{X_s^\eps},Y_s^\eps,\cL_{Y_s^\eps})\psi(X_s^\eps,\cL_{X_s^\eps},Y_s^\eps,\cL_{Y_s^\eps})\no\\ &\quad+\frac{1}{\eps^2}\sL_0(\cL_{X_s^\eps},Y_s^\eps,\cL_{Y_s^\eps}) \psi(X_s^\eps,\cL_{X_s^\eps},Y_s^\eps,\cL_{Y_s^\eps})\Big]\dif s+M_t^1+\frac{1}{\eps}M_t^2+\frac{1}{\eps}M_t^3\no\\
&\quad+\frac{1}{\eps}\int_0^t\T\big((G\sigma_1^*)(X_s^\eps,\cL_{X_s^\eps},Y_s^\eps,\cL_{Y_s^\eps})
\cdot\p_x\p_y\psi(X_s^\eps,\cL_{X_s^\eps},Y_s^\eps,\cL_{Y_s^\eps})\big)\dif s\no\\
&\quad+\tilde\mE\bigg(\int_0^tF(\tilde X^{\eps}_s,\cL_{X_s^\eps},\tilde Y^{\eps}_s,\cL_{Y^{\eps}_s})\cdot\p_\mu\psi(X_s^\eps,\cL_{X_s^\eps},Y_s^\eps,\cL_{Y_s^\eps})(\tilde X^{\eps}_s)\no\\
&\qquad\quad+\frac{1}{2}\T\Big(GG^*(\tilde X^{\eps}_s,\cL_{X_s^\eps},\cL_{Y^{\eps}_s})\cdot\p_{\tilde x}\big[\p_\mu\psi(X_s^\eps,\cL_{X_s^\eps},Y_s^\eps,\cL_{Y_s^\eps})(\tilde X^{\eps}_s)\big]\Big)\no\\
&\qquad\quad+\frac{1}{\eps}c(\tilde X^{\eps}_s,\cL_{X_s^\eps},\tilde Y^{\eps}_s,\cL_{Y^{\eps}_s})\cdot\p_\nu\psi(X_s^\eps,\cL_{X_s^\eps},Y_s^\eps,\cL_{Y_s^\eps})(\tilde Y^{\eps}_s)\dif s\bigg),\label{ff1}
\end{align}
where the operators $\sL_1$ and $\sL_2$ are defined by (\ref{op1}) and (\ref{op2}), respectively, the process ($\tilde X^{\eps}_s,\tilde Y^{\eps}_s$) is a copy of the original process $(X^{\eps}_s,Y^{\eps}_s)$ defined on a copy $(\tilde\Omega,\tilde\sF,\tilde\mP)$ of the original probability space $(\Omega,\sF,\mP)$, and for $i=1,2,3$, $M_t^i$ are martingales defined by
\begin{align*}
M_t^1&:=\int_0^t\p_x\psi(X_s^\eps,\cL_{X_s^\eps},Y_s^\eps,\cL_{Y_s^\eps})\cdot G(X_s^\eps,\cL_{X_s^\eps},\cL_{Y_s^\eps})dW^1_s,\\
M_t^2&:=\int_0^t\p_y\psi(X_s^\eps,\cL_{X_s^\eps},Y_s^\eps,\cL_{Y_s^\eps})\cdot\sigma_1(\cL_{X_s^\eps},Y_s^\eps,\cL_{Y_s^\eps})dW^1_s,\\
M_t^3&:=\int_0^t\p_y\psi(X_s^\eps,\cL_{X_s^\eps},Y_s^\eps,\cL_{Y_s^\eps})\cdot\sigma_2(\cL_{X_s^\eps},Y_s^\eps,\cL_{Y_s^\eps})dW^2_s.
\end{align*}
Multiplying $\eps^2$ from both sides of (\ref{ff1}), taking expectation and in view of  the equation (\ref{pss}), we obtain
\begin{align*}
&\mE\left|\int_0^tf(X_s^\eps,\cL_{X_s^\eps},Y_s^\eps,\cL_{Y_s^\eps})\dif s\right|^2\\
&\leq C_1\Big[\eps^4\,\mE|\psi(\xi,\cL_\xi,\eta,\cL_\eta)|^2+\eps^4\,\mE|\psi(X_t^\eps,\cL_{X_t^\eps},Y_t^\eps,\cL_{Y_t^\eps})|^2\\
&\qquad\quad+\eps^4\,\mE|M_t^1|^2+\eps^2\,\mE|M_t^2|^2+\eps^2\,\mE|M_t^3|^2\Big]\\
&\quad+C_1\,\eps^4\,\mE\bigg|\int_0^t
\sL_1(X_s^\eps,\cL_{X_s^\eps},Y_s^\eps,\cL_{Y_s^\eps})\psi(X_s^\eps,\cL_{X_s^\eps},Y_s^\eps,\cL_{Y_s^\eps})\dif s\bigg|^2\\
&\quad+C_1\,\eps^2\,\mE\bigg|\int_0^t\sL_2(X_s^\eps,\cL_{X_s^\eps},Y_s^\eps,\cL_{Y_s^\eps})  \psi(X_s^\eps,\cL_{X_s^\eps},Y_s^\eps,\cL_{Y_s^\eps})\dif s\bigg|^2\\
&\quad+C_1\,\eps^2\,\mE\bigg|\int_0^t\T\big((G\sigma_1^*)(X_s^\eps,\cL_{X_s^\eps},Y_s^\eps,\cL_{Y_s^\eps})
\cdot\p_x\p_y\psi(X_s^\eps,\cL_{X_s^\eps},Y_s^\eps,\cL_{Y_s^\eps})\big)\dif s\bigg|^2\\
&\quad+C_1\,\eps^4\,\mE\tilde\mE\bigg|\int_0^tF(\tilde X^{\eps}_s,\cL_{X_s^\eps},\tilde Y^{\eps}_s,\cL_{Y^{\eps}_s})\cdot\p_\mu\psi(X_s^\eps,\cL_{X_s^\eps},Y_s^\eps,\cL_{Y_s^\eps})(\tilde X^{\eps}_s)\no\\
&\qquad\quad+\frac{1}{2}\T\Big(GG^*(\tilde X^{\eps}_s,\cL_{X_s^\eps},\cL_{Y^{\eps}_s})\cdot\p_{\tilde x}\big[\p_\mu\psi(X_s^\eps,\cL_{X_s^\eps},Y_s^\eps,\cL_{Y_s^\eps})(\tilde X^{\eps}_s)\big]\Big)\dif s\bigg|^2\no\\
&\quad+C_1\,\eps^2\,\mE\tilde\mE\bigg|\int_0^tc(\tilde X^{\eps}_s,\cL_{X_s^\eps},\tilde Y^{\eps}_s,\cL_{Y^{\eps}_s})\cdot\p_\nu\psi(X_s^\eps,\cL_{X_s^\eps},Y_s^\eps,\cL_{Y_s^\eps})(\tilde Y^{\eps}_s)\dif s\bigg|^2\\
&=:\sum_{i=1}^6\cU_i(\eps).
\end{align*}
By  Lemma \ref{Y1}, we derive that
\begin{align*}
&\mE|\psi(\xi,\cL_\xi,\eta,\cL_\eta)|^2+\mE|\psi(X_t^\eps,\cL_{X_t^\eps},Y_t^\eps,\cL_{Y_t^\eps})|^2\\
&\leq C_1\big(1+\mE|\xi|^2+\mE|\eta|^2+\mE|X_t^\eps|^2+\mE|Y_t^\eps|^2\big)<\infty.
\end{align*}
At the same time, using the Burkholder-Davis-Gundy inequality, we get
\begin{align*}
\mE|M_t^1|^2+\mE|M_t^2|^2+\mE|M_t^3|^2\leq C_1\int_0^t\big(1+\mE|X_s^\eps|^2+\mE|Y_s^\eps|^2\big)\dif s<\infty.
\end{align*}
Consequently, we have
\begin{align}\label{u1}
\cU_1(\eps)\leq C_1\,\eps^2.
\end{align}
Using the assumptions on the coefficients and the regularity of $\psi$ again, and by Lemma \ref{Y1}, we arrive at
\begin{align*}
\sum_{i=2}^6\cU_i(\eps)\leq C_2\,\eps^2\int_0^t\big(1+\mE|X_s^\eps|^4+\mE|Y_s^\eps|^4\big)\dif s\leq C_2\,\eps^2,
\end{align*}
which together with (\ref{u1}) implies the desired result. The proof is thus finished.
\end{proof}

\subsection{Proof of Theorem \ref{main1}}

Throughout this subsection, we assume that the conditions in Theorem \ref{main1} hold. Recall that $X_t^\eps$ and $\bar X_t$ satisfy the McKean-Vlasov equation (\ref{sde}) and (\ref{ave}), respectively. In order to prove the strong convergence of $X_t^\eps$ to $\bar X_t$, we first give the following lemma, which indicates that the averaged function $\bar F$ and $\bar G$ are Lipschitz continuous.

\bl\label{Y3}
Let $\bar F$ and $\bar G$ are defined by (\ref{bF}) and (\ref{sigmaG}), respectively. Then for any $x_1,x_2\in\mR^{d_1}$ and $\mu_1,\mu_2\in\sP_2(\mR^{d_1})$, we have
\begin{align}
&|\bar F(x_1,\mu_1)-\bar F(x_2,\mu_2)|\leq C_0\,\big(|x_1-x_2|+\cW_2(\mu_1,\mu_2)\big),\label{bf2}\\
&\|\bar G(x_1,\mu_1)-\bar G(x_2,\mu_2)\|\leq C_0\,\big(|x_1-x_2|+\cW_2(\mu_1,\mu_2)\big),\label{sigmaG2}
\end{align}
where $C_0>0$ is a constant.
\el

\begin{proof}
By the definition of $\bar F$ and the Lipschitz continuity of $F$, we have
\begin{align}\label{FL}
|\bar F(x_1,\mu_1)-\bar F(x_2,\mu_2)|&=\left|\int_{\mR^{d_2}}F(x_1,\mu_1,y,\zeta^{\mu_1})\zeta^{\mu_1}(\dif y)-\int_{\mR^{d_2}}F(x_2,\mu_2,y,\zeta^{\mu_2})\zeta^{\mu_2}(\dif y)\right|\no\\
&\leq\left|\int_{\mR^{d_2}}\big(F(x_1,\mu_1,y,\zeta^{\mu_1})-F(x_2,\mu_2,y,\zeta^{\mu_2})\big)\zeta^{\mu_1}(\dif y)\right|\no\\
&\quad+\left|\int_{\mR^{d_2}}F(x_2,\mu_2,y,\zeta^{\mu_2})\zeta^{\mu_1}(\dif y)-\int_{\mR^{d_2}}F(x_2,\mu_2,y,\zeta^{\mu_2})\zeta^{\mu_2}(\dif y)\right|\no\\
&\leq C_0\,\big(|x_1-x_2|+\cW_2(\mu_1,\mu_2)+\cW_2(\zeta^{\mu_1},\zeta^{\mu_2})\big).
\end{align}
Using the definition of Wasserstein distance, (\ref{exp}) and Lemma \ref{Y2}, we get
\begin{align*}
\cW_2(\zeta^{\mu_1},\zeta^{\mu_2})^2
&\leq3\cW_2(\zeta^{\mu_1},\cL_{Y^{\mu_1,0}_t})^2+3\cW_2(\zeta^{\mu_2},\cL_{Y^{\mu_2,0}_t})^2
+3\cW_2(\cL_{Y^{\mu_1,0}_t},\cL_{Y^{\mu_2,0}_t})^2\\
&\leq C_0\,e^{-2\lambda_0t}\big(\cW_2(\zeta^{\mu_1},\delta_0)^2+\cW_2(\zeta^{\mu_2},\delta_0)^2\big)+C_0\,\mE|Y^{\mu_1,0}_t-Y^{\mu_2,0}_t|^2\\
&\leq C_0\,e^{-2\lambda_0t}\big(\cW_2(\zeta^{\mu_1},\delta_0)^2+\cW_2(\zeta^{\mu_2},\delta_0)^2\big)+C_0\,\cW_2(\mu_1,\mu_2)^2.
\end{align*}
Letting $t\rightarrow\infty$, we obtain
$$
\cW_2(\zeta^{\mu_1},\zeta^{\mu_2})^2\leq C_0\,\cW_2(\mu_1,\mu_2)^2.
$$
This together with (\ref{FL}) yields (\ref{bf2}). Similarly, we deduce that
\begin{align*}
\|\bar G(x_1,\mu_1)-\bar G(x_2,\mu_2)\|&=\|G(x_1,\mu_1,\zeta^{\mu_1})- G(x_2,\mu_2,\zeta^{\mu_2})\|\\
&\leq C_0\,\big(|x_1-x_2|+\cW_2(\mu_1,\mu_2)+\cW_2(\zeta^{\mu_1},\zeta^{\mu_2})\big),
\end{align*}
which in turn implies that (\ref{sigmaG2}) holds. Thus the proof is completed.
\end{proof}

Now, we are in the position to give:

\begin{proof}[Proof of Theorem \ref{main1}]
In view of (\ref{sde}) and (\ref{ave}), we have
\begin{align*}
X_t^\eps-\bar X_t&=\int_0^t\big[F(X_s^\eps,\cL_{X_s^\eps},Y_s^\eps,\cL_{Y_s^\eps})-\bar F(\bar X_s,\cL_{\bar X_s})\big]\dif s\\
&\quad+\int_0^t\big[G(X_s^\eps,\cL_{X_s^\eps},\cL_{Y_s^\eps})-\bar G(\bar X_s,\cL_{\bar X_s})\big]\dif W^1_s\\
&=\int_0^t\big[F(X_s^\eps,\cL_{X_s^\eps},Y_s^\eps,\cL_{Y_s^\eps})-\bar F(X_s^\eps,\cL_{X_s^\eps})\big]\dif s\\
&\quad+\int_0^t\big[\bar F(X_s^\eps,\cL_{X_s^\eps})-\bar F(\bar X_s,\cL_{\bar X_s})\big]\dif s\\
&\quad+\int_0^t\big[\bar G(X_s^\eps,\cL_{X_s^\eps})-\bar G(\bar X_s,\cL_{\bar X_s})\big]\dif W^1_s\\
&\quad+\int_0^t\big[G(X_s^\eps,\cL_{X_s^\eps},\cL_{Y_s^\eps})-\bar G(X_s^\eps,\cL_{X_s^\eps})\big]\dif W^1_s.
\end{align*}
Taking expectation from both sides of the above equality, we get that there exists a constant $C_0>0$ such that for every $t\in[0,T]$,
\begin{align}\label{X1}
\mE|X_t^\eps-\bar X_t|^2&\leq C_0\mE\left|\int_0^t\big[F(X_s^\eps,\cL_{X_s^\eps},Y_s^\eps,\cL_{Y_s^\eps})-\bar F(X_s^\eps,\cL_{X_s^\eps})\big]\dif s\right|^2\no\\
&\quad+C_0\int_0^t\mE\big|\bar F(X_s^\eps,\cL_{X_s^\eps})-\bar F(\bar X_s,\cL_{\bar X_s})\big|^2\dif s\no\\
&\quad+C_0\int_0^t\mE\big|\bar G(X_s^\eps,\cL_{X_s^\eps})-\bar G(\bar X_s,\cL_{\bar X_s})\big|^2\dif s\no\\
&\quad+ C_0\int_0^t\mE\big|G(X_s^\eps,\cL_{X_s^\eps},\cL_{Y_s^\eps})-\bar G(X_s^\eps,\cL_{X_s^\eps})\big|^2\dif s\no\\
&=:\sJ_1(\eps)+\sJ_2(\eps)+\sJ_3(\eps)+\sJ_4(\eps).
\end{align}
In what follows, we estimate the above four terms one by one. To control the first term, note that by the definition of $\bar F(x,\mu)$, we have
$$
\int_{\mR^{d_2}}\big[F(x,\mu,y,\zeta^\mu)-\bar F(x,\mu)\big]\zeta^\mu(\dif y)=0.
$$
Moreover, by the assumptions that $\sigma_1,\sigma_2,b\in C_b^{(1,1),2,(1,1)}$, $F\in C_b^{2,(1,1),2,(1,1)}$ and Lemma \ref{reh}, we have
$$
F(x,\mu,y,\nu)-\bar F(x,\mu)\in C_b^{2,(1,1),2,(1,1)}(\mR^{d_1}\times\sP_2(\mR^{d_1})\times\mR^{d_2}\times\sP_2(\mR^{d_2})).
$$
This together with  the assumption $\p_xF(x,\mu,\cdot,\nu)\in C_b^1(\mR^{d_2})$ and Lemma \ref{flu} yields
\begin{align}\label{X2}
\sJ_1(\eps)\leq C_1\,\eps^2.
\end{align}
For the second and third terms, using Lemma \ref{Y3} we deduce that
\begin{align}\label{X3}
\sJ_2(\eps)+\sJ_3(\eps)&\leq C_2\int_0^t\big[\mE|X_s^\eps-\bar X_s|^2+\cW_2(\cL_{X_s^\eps},\cL_{\bar X_s})^2\big]\dif s\no\\
&\leq C_3\int_0^t\mE|X_s^\eps-\bar X_s|^2\dif s.
\end{align}
As for $\sJ_4(\eps)$, by the definition of $\bar G(x,\mu)$ and the assumption on $G$, we have
\begin{align}\label{X4}
\sJ_4(\eps)
&\leq C_4\int_0^t\cW_2(\cL_{Y_{s/\eps^2}^{\mu,\eta}}|_{\mu=\cL_{X^\eps_s}},\zeta^{\cL_{X_s^\eps}})^2\dif s+C_4\int_0^t\cW_2(\cL_{Y_s^\eps},\cL_{Y_{s/\eps^2}^{\mu,\eta}}|_{\mu=\cL_{X^\eps_s}})^2\dif s\no\\
&=:\sJ_{4,1}(\eps)+\sJ_{4,2}(\eps),
\end{align}
where $Y_t^{\mu,\eta}$ is the unique strong solution of the equation (\ref{sde1}).
Applying (\ref{exp}), (\ref{zeta1}) and Lemma \ref{Y1}, we get
\begin{align}\label{X6}
\sJ_{4,1}(\eps)&=C_4\int_0^t\cW_2(\cL_{Y_{s/\eps^2}^{\mu,\eta}}|_{\mu=\cL_{X^\eps_s}},\zeta^{\cL_{X_s^\eps}})^2\dif s\no\\
&\leq C_4\int_0^te^{-\frac{2\lambda_0s}{\eps^2}}\cdot\cW_2(\zeta^{\cL_{X^\eps_s}},\nu)^2\dif s\no\\
&\leq C_4\int_0^te^{-\frac{2\lambda_0s}{\eps^2}}\cdot\big(1+\cW_2(\cL_{X^\eps_s},\delta_0)^2\big)\dif s\leq C_4\,\eps^2.
\end{align}
To control $\sJ_{4,2}(\eps)$, for every $t\geq0$, we let $\hat Y^\eps_t=Y^\eps_{\eps^2t}$. Then it is easy to see that $\hat Y^\eps_t$ satisfies the following equation:
\begin{align*}
\hat Y^\eps_t&=\eta+\eps\int_0^tc(X^{\eps}_{\eps^2s},\cL_{X^\eps_{\eps^2s}},\hat{Y}^{\eps}_s,\cL_{\hat{Y}^{\eps}_s})\dif s+\int_0^tb(\cL_{X^\eps_{\eps^2s}},\hat{Y}^{\eps}_s,\cL_{\hat{Y}^{\eps}_s})\dif s\\
&\quad+\int_0^t\sigma_1(\cL_{X^\eps_{\eps^2s}},\hat Y^{\eps}_s,\cL_{\hat Y^{\eps}_s})\dif \hat W_s^1+\int_0^t\sigma_2(\cL_{X^\eps_{\eps^2s}},\hat Y^{\eps}_s,\cL_{\hat Y^{\eps}_s})\dif \hat W_s^2.
\end{align*}
Note that
\begin{align}\label{42}
\cW_2(\cL_{Y_t^\eps},\cL_{Y_{t/\eps^2}^{\mu,\eta}}|_{\mu=\cL_{X^\eps_t}})^2
=\cW_2(\cL_{\hat Y^\eps_{t/\eps^2}},\cL_{Y_{t/\eps^2}^{\mu,\eta}}|_{\mu=\cL_{X^\eps_t}})^2
\leq\Big[\mE\big|\hat Y^\eps_{t/\eps^2}-Y_{t/\eps^2}^{\mu,\eta}\big|^2\Big]_{\mu=\cL_{X^\eps_t}}.
\end{align}
By It\^{o}'s formula and the assumption {\bf (H$^{\sigma,b}$)}, we deduce that
 for every $t>0$,
\begin{align*}
\dif \mE|\hat Y^\eps_t-Y_t^{\mu,\eta}|^2&=\mE\big[2\langle\hat Y^\eps_t-Y_t^{\mu,\eta},b(\cL_{X^\eps_{\eps^2t}},\hat{Y}^{\eps}_t,\cL_{\hat{Y}^{\eps}_t})
-b(\mu,Y_t^{\mu,\eta},\cL_{Y_t^{\mu,\eta}})\rangle\\
&\quad+\|\sigma_1(\cL_{X^\eps_{\eps^2t}},\hat{Y}^{\eps}_t,\cL_{\hat{Y}^{\eps}_t})
-\sigma_1(\mu,Y_t^{\mu,\eta},\cL_{Y_t^{\mu,\eta}})\|^2\\
&\quad+\|\sigma_2(\cL_{X^\eps_{\eps^2t}},\hat{Y}^{\eps}_t,\cL_{\hat{Y}^{\eps}_t})
-\sigma_2(\mu,Y_t^{\mu,\eta},\cL_{Y_t^{\mu,\eta}})\|^2\big]\dif t\\
&\quad+\mE\big[2\langle\hat Y^\eps_t-Y_t^{\mu,\eta},\eps c(X^\eps_{\eps^2t},\cL_{X^\eps_{\eps^2t}},\hat{Y}^{\eps}_t,\cL_{\hat{Y}^{\eps}_t})\rangle\big]\dif t\\
&\leq-(c_2-c_1)\,\mE|\hat Y^\eps_t-Y_t^{\mu,\eta}|^2\dif t+C_4\,\cW_2(\cL_{X^\eps_{\eps^2t}},\mu)^2\dif t\\
&\quad+C_4\,\eps^2\big(1+\mE|X_{\eps^2t}^\eps|^2+\mE|Y_{\eps^2t}^\eps|^2\big)\dif t,
\end{align*}
which together with the comparison theorem implies
\begin{align*}
\mE|\hat Y^\eps_t-Y_t^{\mu,\eta}|^2&\leq C_4\int_0^te^{-(c_2-c_1)(t-s)}\cdot\cW_2(\cL_{X^\eps_{\eps^2s}},\mu)^2\dif s\\
&\quad+C_4\,\eps^2\int_0^te^{-(c_2-c_1)(t-s)}\cdot\big(1+\mE|X_{\eps^2s}^\eps|^2+\mE|Y_{\eps^2s}^\eps|^2\big)\dif s.
\end{align*}
As a result, we have
\begin{align*}
&\Big[\mE\big|\hat Y^\eps_{t/\eps^2}-Y_{t/\eps^2}^{\mu,\eta}\big|^2\Big]_{\mu=\cL_{X^\eps_t}}\\
&\leq C_4\int_0^{\frac{t}{\eps^2}}e^{-(c_2-c_1)\big(\frac{t}{\eps^2}-s\big)}\cdot\cW_2(\cL_{X^\eps_{\eps^2s}},\cL_{X^\eps_t})^2\dif s\\
&\quad+C_4\,\eps^2\int_0^{\frac{t}{\eps^2}}e^{-(c_2-c_1)\big(\frac{t}{\eps^2}-s\big)}\cdot\big(1+\mE|X_{\eps^2s}^\eps|^2+\mE|Y_{\eps^2s}^\eps|^2\big)\dif s\\
&\leq C_4\,\frac{1}{\eps^2}\int_0^te^{-(c_2-c_1)\frac{t-s}{\eps^2}}\cdot\mE|X^\eps_s-X^\eps_t|^2\dif s\\
&\quad+C_4\int_0^te^{-(c_2-c_1)\frac{t-s}{\eps^2}}\cdot\big(1+\mE|X_s^\eps|^2+\mE|Y_s^\eps|^2\big)\dif s\\
&\leq C_4\int_0^te^{-(c_2-c_1)\frac{t-s}{\eps^2}}\cdot \frac{t-s}{\eps^2}\dif s+C_4\,\eps^2\leq C_4\,\eps^2,
\end{align*}
where the constant $C_4$ is independent of $\eps$. Taking this back into (\ref{42}), we arrive at
\begin{align}\label{X5}
\sJ_{4,2}(\eps)\leq C_4\,\eps^2.
\end{align}
Substituting (\ref{X5}) and (\ref{X6}) into (\ref{X4}) gives
$$
\sJ_4(\eps)\leq C_4\,\eps^2.
$$
This together with (\ref{X1}), (\ref{X2}) and (\ref{X3}) yields
\begin{align*}
\mE|X_t^\eps-\bar X_t|^2&\leq C_5\,\eps^2+C_5\int_0^t\mE|X_s^\eps-\bar X_s|^2\dif s,
\end{align*}
which in turn implies the desired assertion by Gronwall's inequality.
\end{proof}

\section{Functional central limit type theorem}

In this section, we study the functional central limit type theorem for the system (\ref{sde}). Recall that assumption (\ref{G}) holds,
i.e.,the diffusion coefficient $G$ in the slow equation does not depend on the distribution of the fast motion. Note that in this case, we get
$$
\bar G(x,\mu)=G(x,\mu).
$$
We shall first derive some weak fluctuation estimates in Subsection 5.1. Then we give the proof of Theorem \ref{main2} in Subsection 5.2.

\subsection{Weak Fluctuation estimates}

Recall that $Z_t^\eps$ and $\bar Z_t$ are defined by (\ref{z}) and (\ref{lim}), respectively.
By definition, we can write
\begin{align*}
\dif Z_t^\eps
&=\frac{1}{\eps} \delta F(X_t^\eps,\cL_{X_t^\eps},Y_t^\eps,\cL_{Y_t^\eps})\dif t\\
&\quad+\Big(\frac{1}{\eps}\big[\bar F(X_t^\eps,\cL_{X_t^\eps})-\bar F(\bar X_t,\cL_{\bar X_t})\big]\dif t+\frac{1}{\eps}\big[G(X_t^\eps,\cL_{X_t^\eps})-G(\bar X_t,\cL_{\bar X_t})\big]\dif W_t^1\Big),
\end{align*}
where $\delta F$ is defined by (\ref{pss2}). For the sake of convenience, we define
\begin{align}\label{op3}
\sL_3:=\sL_3(x,\mu,y,\nu):=\sum_{i=1}^{d_1}\big[F_i(x,\mu,y,\nu)-\bar F_i(x,\mu)\big]\frac{\p}{\p z_i},
\end{align}
and
\begin{align}\label{op4}
\sL_4^\eps&:=\sL_4^\eps(x,\mu,\bar x,\bar\mu):=\frac{1}{\eps}\sum_{i=1}^{d_1}\big[\bar F_i(x,\mu)-\bar F_i(\bar x,\bar\mu)\big]\frac{\p}{\p z_i}\no\\
&\quad+\frac{1}{2\eps^2}\sum_{i,j=1}^{d_1}\Big(\big[G(x,\mu)-G(\bar x,\bar\mu)\big]\big[G(x,\mu)-G(\bar x,\bar\mu)\big]^*\Big)_{ij}\frac{\p^2}{\p z_i\p z_j}\no\\
&\quad+\frac{1}{\eps}\sum_{i,j=1}^{d_1}\Big(G(x,\mu)\big[G(x,\mu)-G(\bar x,\bar\mu)\big]^*\Big)_{ij}\frac{\p^2}{\p x_i\p z_j}.
\end{align}
Let $f(t,x,\mu,y,\nu,z,\pi)$ be a function satisfying the centering condition, i.e., for every fixed $(t,x,\mu,z,\pi)\in\mR_+\times\mR^{d_1}\times\sP_2(\mR^{d_1})\times\mR^{d_1}\times\sP_2(\mR^{d_1})$,
\begin{align}\label{cen1}
\int_{\mR^{d_2}}f(t,x,\mu,y,\zeta^\mu,z,\pi)\zeta^\mu(\dif y)=0.
\end{align}
Then, we consider the following Poisson equation:
\begin{align}\label{pss1}
\sL_0(\mu,y,\nu)\psi(t,x,\mu,y,\nu,z,\pi)=-f(t,x,\mu,y,\nu,z,\pi),
\end{align}
where $(t,x,\mu,z,\pi)$ are regarded as parameters. We have the following  fluctuation estimate for the process $f(t,X_t^\eps,\cL_{X_t^\eps},Y_t^\eps,\cL_{Y_t^\eps},Z_t^\eps,\cL_{Z_t^\eps})$.

\bl
Let {\bf (H$^{\sigma,b}$)} hold. Assume that $F,G$ and $c$ are Lipschitz continuous and $\sigma_1,\sigma_2,b\in C_b^{(1,1),2,(1,1)}(\sP_2(\mR^{d_1})\times\mR^{d_2}\times\sP_2(\mR^{d_2}))$. Then for $f\in C_b^{1,2,(1,1),2,(1,1),2,(1,1)}(\mR_+\times\mR^{d_1}\times\sP_2(\mR^{d_1})\times\mR^{d_2}\times\sP_2(\mR^{d_2})\times\mR^{d_1}\times\sP_2(\mR^{d_1}))$ satisfying $\p_xf,\p_zf(t,x,\mu,\cdot,\nu,z,\pi)
\in C_b^1(\mR^{d_2})$ and  (\ref{cen1}), we have
\begin{align}\label{es1}
&\mE\left(\int_0^tf(s,X_s^\eps,\cL_{X_s^\eps},Y_s^\eps,\cL_{Y_s^\eps},Z_s^\eps,\cL_{Z_s^\eps})\dif s\right)\no\\
&\leq C_0\,\eps^2+C_0\,\eps\bigg[\mE\bigg(\int_0^t
\sL_2\psi(s,X_s^\eps,\cL_{X_s^\eps},Y_s^\eps,\cL_{Y_s^\eps},Z_s^\eps,\cL_{Z_s^\eps})\dif s\bigg)\no\\
&\quad+\mE\bigg(\int_0^t\sL_3\psi(s,X_s^\eps,\cL_{X_s^\eps},Y_s^\eps,\cL_{Y_s^\eps},Z_s^\eps,\cL_{Z_s^\eps})
\dif s\bigg)\no\\
&\quad+\mE\bigg(\int_0^t\T\big((G\sigma_1^*)(X_s^\eps,\cL_{X_s^\eps},Y_s^\eps,\cL_{Y_s^\eps})
\cdot\p_x\p_y\psi(s,X_s^\eps,\cL_{X_s^\eps},Y_s^\eps,\cL_{Y_s^\eps},Z_s^\eps,\cL_{Z_s^\eps})\big)\dif s\bigg)\no\\
&\quad+\mE\bigg(\int_0^t\T\bigg(\sigma_1(\cL_{X_s^\eps},Y_s^\eps,\cL_{Y_s^\eps})\frac{[G(X_s^\eps,\cL_{X_s^\eps})- G(\bar X_s,\cL_{\bar X_s})]^*}{\eps}\no\\
&\qquad\qquad\times\p_y\p_z\psi(s,X_s^\eps,\cL_{X_s^\eps},Y_s^\eps,\cL_{Y_s^\eps},Z_s^\eps,\cL_{Z_s^\eps})\bigg)\dif s\bigg)\no\\
&\quad+\mE\tilde\mE\bigg(\int_0^tc(\tilde X^{\eps}_s,\cL_{X_s^\eps},\tilde Y^{\eps}_s,\cL_{Y^{\eps}_s})\cdot\p_\nu\psi(s,X_s^\eps,\cL_{X_s^\eps},Y_s^\eps,\cL_{Y_s^\eps},Z_s^\eps,\cL_{Z_s^\eps})(\tilde Y^{\eps}_s)\dif s\bigg)\no\\
&\quad+\mE\tilde\mE\bigg(\int_0^t\delta F(\tilde X^{\eps}_s,\cL_{X_s^\eps},\tilde Y^{\eps}_s,\cL_{Y^{\eps}_s})\cdot\p_\pi\psi(s,X_s^\eps,\cL_{X_s^\eps},Y_s^\eps,\cL_{Y_s^\eps},Z_s^\eps,\cL_{Z_s^\eps})(\tilde Z^{\eps}_s)\dif s\bigg)\bigg],
\end{align}
where $\sL_2$ and $\sL_3$ are given by (\ref{op2}) and (\ref{op3}), respectively, and $C_0>0$ is a constant independent of $\eps$.
\el

\br
Note that under the above assumptions and according to Theorem \ref{pfs}, we have $\psi\in C_b^{1,2,(1,1),2,(1,1),2,(1,1)}$ and $\p_x\psi,\p_z\psi(t,x,\mu,\cdot,\nu,z,\pi)
\in C_b^1(\mR^{d_2})$. Then we can compute that
\begin{align*}
&\mE\bigg(\int_0^t\T\bigg(\sigma_1(\cL_{X_s^\eps},Y_s^\eps,\cL_{Y_s^\eps})\frac{[G(X_s^\eps,\cL_{X_s^\eps})- G(\bar X_s,\cL_{\bar X_s})]^*}{\eps}\\
&\qquad\qquad\times\p_y\p_z\psi(s,X_s^\eps,\cL_{X_s^\eps},Y_s^\eps,\cL_{Y_s^\eps},Z_s^\eps,\cL_{Z_s^\eps})\bigg)\dif s\bigg)\\
&\leq C_0\int_0^t\mE\Big[\big(1+|Y_s^\eps|+\cW_2(\cL_{X_s^\eps},\delta_0)+\cW_2(\cL_{Y_s^\eps},\delta_0)\big)
\big(|Z_s^\eps|+(\mE|Z_s^\eps|^2)^{1/2}\big)\Big]\dif s\\
&\leq C_0\int_0^t\big(1+\mE|X_s^\eps|^2+\mE|Y_s^\eps|^2+\mE|Z_s^\eps|^2\big)\dif s<\infty.
\end{align*}
Meanwhile, the other terms on the right hand side of (\ref{es1}) can be controlled similarly. Therefore, we have
\begin{align}\label{es2}
\mE\left(\int_0^tf(s,X_s^\eps,\cL_{X_s^\eps},Y_s^\eps,\cL_{Y_s^\eps},Z_s^\eps,\cL_{Z_s^\eps})\dif s\right)\leq C_0\,\eps.
\end{align}
However, the homogenization effects of the terms involving expectations in (\ref{es1}) will appear when we investigate the central limit theorem, so we keep them for later use.
\er
\begin{proof}
Let $\psi$ be the solution of the Poisson equation (\ref{pss1}). Then by It\^o's formula, we deduce that
\begin{align}
&\psi(t,X_t^\eps,\cL_{X_t^\eps},Y_t^\eps,\cL_{Y_t^\eps},Z_t^\eps,\cL_{Z_t^\eps}) =\psi(0,\xi,\cL_{\xi},\eta,\cL_\eta,0,\delta_0)\no\\
&\quad+\int_0^t\Big(\p_s+\sL_1+\sL_4^\eps\Big)\psi(s,X_s^\eps,\cL_{X_s^\eps},Y_s^\eps,\cL_{Y_s^\eps},Z_s^\eps,\cL_{Z_s^\eps})\dif s\no\\
&\quad+\frac{1}{\eps}\int_0^t\T\big((G\sigma_1^*)(X_s^\eps,\cL_{X_s^\eps},Y_s^\eps,\cL_{Y_s^\eps})
\cdot\p_x\p_y\psi(s,X_s^\eps,\cL_{X_s^\eps},Y_s^\eps,\cL_{Y_s^\eps},Z_s^\eps,\cL_{Z_s^\eps})\big)\dif s\no\\
&\quad+M_t^1+\frac{1}{\eps}M_t^2+\frac{1}{\eps}M_t^3+\frac{1}{\eps}M_t^4\no\\
&\quad+\frac{1}{\eps}\int_0^t\Big(\sL_2+\sL_3\Big)
\psi(s,X_s^\eps,\cL_{X_s^\eps},Y_s^\eps,\cL_{Y_s^\eps},Z_s^\eps,\cL_{Z_s^\eps})\dif s\no\\
&\quad+\frac{1}{\eps^2}\int_0^t\sL_0(\cL_{X_s^\eps},Y_s^\eps,\cL_{Y_s^\eps}) \psi(s,X_s^\eps,\cL_{X_s^\eps},Y_s^\eps,\cL_{Y_s^\eps},Z_s^\eps,\cL_{Z_s^\eps})\dif s\no\\
&\quad+\frac{1}{\eps^2}\int_0^t\T\big(\sigma_1(\cL_{X_s^\eps},Y_s^\eps,\cL_{Y_s^\eps})[G(X_s^\eps,\cL_{X_s^\eps})- G(\bar X_s,\cL_{\bar X_s})]^*\no\\
&\qquad\qquad\times\p_y\p_z\psi(s,X_s^\eps,\cL_{X_s^\eps},Y_s^\eps,\cL_{Y_s^\eps},Z_s^\eps,\cL_{Z_s^\eps})\big)\dif s\no\\
&\quad+\tilde\mE\bigg(\int_0^tF(\tilde X^{\eps}_s,\cL_{X_s^\eps},\tilde Y^{\eps}_s,\cL_{Y^{\eps}_s})\cdot\p_\mu\psi(s,X_s^\eps,\cL_{X_s^\eps},Y_s^\eps,\cL_{Y_s^\eps},Z_s^\eps,\cL_{Z_s^\eps})(\tilde X^{\eps}_s)\no\\
&\qquad\quad+\frac{1}{2}\T\Big(GG^*(\tilde X^{\eps}_s,\cL_{X_s^\eps})\cdot\p_{\tilde x}\big[\p_\mu\psi(s,X_s^\eps,\cL_{X_s^\eps},Y_s^\eps,\cL_{Y_s^\eps},Z_s^\eps,\cL_{Z_s^\eps})(\tilde X^{\eps}_s)\big]\Big)\no\\
&\qquad\quad+\frac{1}{\eps}c(\tilde X^{\eps}_s,\cL_{X_s^\eps},\tilde Y^{\eps}_s,\cL_{Y^{\eps}_s})\cdot\p_\nu\psi(s,X_s^\eps,\cL_{X_s^\eps},Y_s^\eps,\cL_{Y_s^\eps},Z_s^\eps,\cL_{Z_s^\eps})(\tilde Y^{\eps}_s)\no\\
&\qquad\quad+\frac{1}{\eps}[F(\tilde X^{\eps}_s,\cL_{X_s^\eps},\tilde Y^{\eps}_s,\cL_{Y^{\eps}_s})-\bar F(\tilde X^{\eps}_s,\cL_{X_s^\eps})]\no\\
&\qquad\qquad\quad\times \p_\pi\psi(s,X_s^\eps,\cL_{X_s^\eps},Y_s^\eps,\cL_{Y_s^\eps},Z_s^\eps,\cL_{Z_s^\eps})(\tilde Z^{\eps}_s)\no\\
&\qquad\quad+\frac{1}{\eps}[\bar F(\tilde X^{\eps}_s,\cL_{X_s^\eps})-\bar F(\tilde{\bar X}_s,\cL_{\bar X_s})]\cdot\p_\pi\psi(s,X_s^\eps,\cL_{X_s^\eps},Y_s^\eps,\cL_{Y_s^\eps},Z_s^\eps,\cL_{Z_s^\eps})(\tilde Z^{\eps}_s)\no\\
&\qquad\quad+\frac{1}{2\eps^2}\T\Big([G(\tilde X^{\eps}_s,\cL_{X_s^\eps})-G(\tilde{\bar X}_s,\cL_{\bar X_s})][G(\tilde X^{\eps}_s,\cL_{X_s^\eps})-G(\tilde{\bar X}_s,\cL_{\bar X_s})]^*\no\\
&\qquad\qquad\times\p_{\tilde z}\p_\pi\psi(s,X_s^\eps,\cL_{X_s^\eps},Y_s^\eps,\cL_{Y_s^\eps},Z_s^\eps,\cL_{Z_s^\eps})(\tilde Z^{\eps}_s)\Big)\dif s\bigg),\label{f1}
\end{align}
where the operators $\sL_0$, $\sL_1$, $\sL_2$, $\sL_3$ and $\sL_4^\eps$ are given by (\ref{op}), (\ref{op1}), (\ref{op2}), (\ref{op3}) and (\ref{op4}), respectively, the process ($\tilde X^{\eps}_s,\tilde {\bar X}_s,\tilde Y^{\eps}_s,\tilde Z^{\eps}_s$) is a copy of the original process $(X^{\eps}_s,\bar X_s,Y^{\eps}_s,Z_s^{\eps})$ defined on a copy $(\tilde\Omega,\tilde\sF,\tilde\mP)$ of the original probability space $(\Omega,\sF,\mP)$, and for $i=1,2,3,4$, $M_t^i$ are martingales defined by
\begin{align*}
M_t^1&:=\int_0^t\p_x\psi(s,X_s^\eps,\cL_{X_s^\eps},Y_s^\eps,\cL_{Y_s^\eps},Z_s^\eps,\cL_{Z_s^\eps})\cdot G(X_s^\eps,\cL_{X_s^\eps})\dif W_s^1,\\
M_t^2&:=\int_0^t\p_y\psi(s,X_s^\eps,\cL_{X_s^\eps},Y_s^\eps,\cL_{Y_s^\eps},Z_s^\eps,\cL_{Z_s^\eps})\cdot \sigma_1(\cL_{X_s^\eps},Y_s^\eps,\cL_{Y_s^\eps})\dif W_s^1,\\
M_t^3&:=\int_0^t\p_y\psi(s,X_s^\eps,\cL_{X_s^\eps},Y_s^\eps,\cL_{Y_s^\eps},Z_s^\eps,\cL_{Z_s^\eps})\cdot \sigma_2(\cL_{X_s^\eps},Y_s^\eps,\cL_{Y_s^\eps})\dif W_s^2,\\
M_t^4&:=\int_0^t\p_z\psi(s,X_s^\eps,\cL_{X_s^\eps},Y_s^\eps,\cL_{Y_s^\eps},Z_s^\eps,\cL_{Z_s^\eps})\cdot
\big[G(X_s^\eps,\cL_{X_s^\eps})-G(\bar X_s,\cL_{\bar X_s})\big]\dif W_s^1.
\end{align*}
Multiplying $\eps^2$ and taking expectation on both sides of (\ref{f1}), and in view of (\ref{pss1}), we obtain
\begin{align*}
&\mE\left(\int_0^tf(s,X_s^\eps,\cL_{X_s^\eps},Y_s^\eps,\cL_{Y_s^\eps},Z_s^\eps,\cL_{Z_s^\eps})\dif s\right)\no\\
&=\eps^2\mE\big[\psi(0,\xi,\cL_{\xi},\eta,\cL_\eta,0,\delta_0)
-\psi(t,X_t^\eps,\cL_{X_t^\eps},Y_t^\eps,\cL_{Y_t^\eps},Z_t^\eps,\cL_{Z_t^\eps})\big]\\
&\quad+\eps^2\mE\bigg(\int_0^t\Big(\p_s+\sL_1+\sL_4^\eps\Big)\psi(s,X_s^\eps,\cL_{X_s^\eps},Y_s^\eps,\cL_{Y_s^\eps},Z_s^\eps,\cL_{Z_s^\eps})\dif s\bigg)\no\\
&\quad+\eps^2\mE\tilde\mE\bigg(\int_0^tF(\tilde X^{\eps}_s,\cL_{X_s^\eps},\tilde Y^{\eps}_s,\cL_{Y^{\eps}_s})\cdot\p_\mu\psi(s,X_s^\eps,\cL_{X_s^\eps},Y_s^\eps,\cL_{Y_s^\eps},Z_s^\eps,\cL_{Z_s^\eps})(\tilde X^{\eps}_s)\no\\
&\qquad\quad+\frac{1}{2}\T\Big(GG^*(\tilde X^{\eps}_s,\cL_{X_s^\eps})\cdot\p_{\tilde x}\big[\p_\mu\psi(s,X_s^\eps,\cL_{X_s^\eps},Y_s^\eps,\cL_{Y_s^\eps},Z_s^\eps,\cL_{Z_s^\eps})(\tilde X^{\eps}_s)\big]\bigg)\dif s\bigg)\no\\
&\quad+\eps^2\mE\tilde\mE\bigg(\int_0^t\frac{[\bar F(\tilde X^{\eps}_s,\cL_{X_s^\eps})-\bar F(\tilde{\bar X}_s,\cL_{\bar X_s})]}{\eps}\no\\
&\qquad\qquad\quad\times\p_\pi\psi(s,X_s^\eps,\cL_{X_s^\eps},Y_s^\eps,\cL_{Y_s^\eps},Z_s^\eps,\cL_{Z_s^\eps})(\tilde Z^{\eps}_s)\dif s\bigg)\no\\
&\quad+\eps^2\mE\tilde\mE\bigg(\int_0^t\frac{1}{2}\T\bigg(\frac{[G(\tilde X^{\eps}_s,\cL_{X_s^\eps})-G(\tilde{\bar X}_s,\cL_{\bar X_s})][G(\tilde X^{\eps}_s,\cL_{X_s^\eps})-G(\tilde{\bar X}_s,\cL_{\bar X_s})]^*}{\eps^2}\no\\
&\qquad\qquad\times\p_{\tilde z}\p_\pi\psi(s,X_s^\eps,\cL_{X_s^\eps},Y_s^\eps,\cL_{Y_s^\eps},Z_s^\eps,\cL_{Z_s^\eps})(\tilde Z^{\eps}_s)\bigg)\dif s\bigg)\no\\
&\quad+\eps\bigg[\mE\bigg(\int_0^t\sL_2(X_s^\eps,\cL_{X_s^\eps},Y_s^\eps,\cL_{Y_s^\eps})\psi(s,X_s^\eps,\cL_{X_s^\eps},Y_s^\eps,\cL_{Y_s^\eps},Z_s^\eps,\cL_{Z_s^\eps})\dif s\bigg)\no\\
&\quad+\mE\bigg(\int_0^t\sL_3(X_s^\eps,\cL_{X_s^\eps},Y_s^\eps,\cL_{Y_s^\eps})
\psi(s,X_s^\eps,\cL_{X_s^\eps},Y_s^\eps,\cL_{Y_s^\eps},Z_s^\eps,\cL_{Z_s^\eps})\dif s\bigg)\no\\
&\quad+\mE\bigg(\int_0^t\T\big((G\sigma_1^*)(X_s^\eps,\cL_{X_s^\eps},Y_s^\eps,\cL_{Y_s^\eps})
\cdot\p_x\p_y\psi(s,X_s^\eps,\cL_{X_s^\eps},Y_s^\eps,\cL_{Y_s^\eps},Z_s^\eps,\cL_{Z_s^\eps})\big)\dif s\bigg)\no\\
&\quad+\mE\bigg(\int_0^t\T\bigg(\sigma_1(\cL_{X_s^\eps},Y_s^\eps,\cL_{Y_s^\eps})\frac{[G(X_s^\eps,\cL_{X_s^\eps})- G(\bar X_s,\cL_{\bar X_s})]^*}{\eps}\no\\
&\qquad\qquad\times\p_y\p_z\psi(s,X_s^\eps,\cL_{X_s^\eps},Y_s^\eps,\cL_{Y_s^\eps},Z_s^\eps,\cL_{Z_s^\eps})\bigg)\dif s\bigg)\no\\
&\quad+\mE\tilde\mE\bigg(\int_0^tc(\tilde X^{\eps}_s,\cL_{X_s^\eps},\tilde Y^{\eps}_s,\cL_{Y^{\eps}_s})\cdot\p_\nu\psi(s,X_s^\eps,\cL_{X_s^\eps},Y_s^\eps,\cL_{Y_s^\eps},Z_s^\eps,\cL_{Z_s^\eps})(\tilde Y^{\eps}_s)\dif s\bigg)\no\\
&\quad+\mE\tilde\mE\bigg(\int_0^t\delta F(\tilde X^{\eps}_s,\cL_{X_s^\eps},\tilde Y^{\eps}_s,\cL_{Y^{\eps}_s})\cdot\p_\pi\psi(s,X_s^\eps,\cL_{X_s^\eps},Y_s^\eps,\cL_{Y_s^\eps},Z_s^\eps,\cL_{Z_s^\eps})(\tilde Z^{\eps}_s)\dif s\bigg)\bigg]\no\\
&=:\sum_{i=1}^{5}\cV_i(\eps)+\eps\,\cV_6(\eps).
\end{align*}
Let us first deal with the term involving $\sL_4^\eps$. Using the assumption on $F$ and by Lemma \ref{Y3}, we have that $\bar F$ satisfies the Lipschitz condition. This, together with the assumption on $G$ and the regularity of $\psi$, yields that for some $C_1>0$,
\begin{align}\label{terl4}
&\mE\bigg(\int_0^t\sL_4^\eps\psi(s,X_s^\eps,\cL_{X_s^\eps},Y_s^\eps,\cL_{Y_s^\eps},Z_s^\eps,\cL_{Z_s^\eps})\dif s\bigg)\no\\
&\leq C_1\mE\bigg(\int_0^t\bigg[\frac{|\bar F(X_s^\eps,\sL_{X_s^\eps})-\bar F(\bar X_s,\sL_{\bar X_s})|}{\eps}\no\\
&\qquad+\frac{\| G(X_s^\eps,\sL_{X_s^\eps})-G(\bar X_s,\sL_{\bar X_s})\|^2}{\eps^2}\no\\
&\qquad+\big(1+|X_s^\eps|+\cW_2(\sL_{X_s^\eps},\delta_0)\big)\frac{\| G(X_s^\eps,\sL_{X_s^\eps})-G(\bar X_s,\sL_{\bar X_s})\|}{\eps}\bigg]\dif s\bigg)\no\\
&\leq C_1\int_0^t\big(1+\mE|X_s^\eps|^2+\mE|Z_s^\eps|^2\big)\dif s<\infty.
\end{align}
Similarly, we have
\begin{align*}
\cV_4(\eps)+\cV_5(\eps)&\leq C_2\,\eps^2\,\mE\tilde\mE\bigg(\int_0^t\bigg[\frac{|\bar F(\tilde X_s^\eps,\sL_{X_s^\eps})-\bar F(\tilde{\bar X}_s,\sL_{\bar X_s})|}{\eps}\\
&\quad+\frac{\| G(\tilde X_s^\eps,\sL_{X_s^\eps})-G(\tilde{\bar X}_s,\sL_{\bar X_s})\|^2}{\eps^2}\bigg]\dif s\bigg)\no\\
&\leq C_2\,\eps^2\,\int_0^t\big(1+\mE|Z_s^\eps|^2\big)\dif s\leq C_2\,\eps^2.
\end{align*}
This together with (\ref{terl4}) and the assumptions on coefficients implies that
\begin{align*}
\sum_{i=1}^{5}\cV_i(\eps)\leq C_3\,\eps^2+C_3\,\eps^2\int_0^t\big(1+\mE|X_s^\eps|^2+\mE|Y_s^\eps|^2\big)\dif s\leq C_3\,\eps^2,
\end{align*}
which in turn yields the desired conclusion. Thus the proof is completed.
\end{proof}

\subsection{Proof of Theorem \ref{main2}}

Throughout this subsection, we assume that the conditions in Theorem \ref{main2} hold. Let $\bar X_t^{s,\xi}$ be the unique solution of the equation (\ref{ave}) starting from the initial data $\xi\in L^2(\Omega)$ at time $s$, and $\bar Z_t^{s,\xi,\vartheta}$ be the unique solution of the equation (\ref{lim}) with the initial value $\vartheta\in L^2(\Omega)$ at time $s$. Namely, for $t\geq s$,
\begin{align*}
\dif \bar X_t^{s,\xi}=\bar F(\bar X_t^{s,\xi},\cL_{\bar X_t^{s,\xi}})\dif t+G(\bar X_t^{s,\xi},\cL_{\bar X_t^{s,\xi}})\dif W_t^1,\qquad\bar X_s^{s,\xi}=\xi,
\end{align*}
and
\begin{align*}
\dif \bar Z_t^{s,\xi,\vartheta}&=\p_x\bar F(\bar X_t^{s,\xi},\cL_{\bar X_t^{s,\xi}})\bar Z_t^{s,\xi,\vartheta}\dif t+\tilde\mE\big[\p_\mu\bar F(\bar X_t^{s,\xi},\cL_{\bar X_t^{s,\xi}})(\tilde{\bar X}_t^{s,\tilde\xi})\tilde{\bar Z}_t^{s,\tilde\xi,\tilde\vartheta}\big]\dif t\\
&\quad+\overline{c\cdot\p_y\Phi}(\bar X_t^{s,\xi},\cL_{\bar X_t^{s,\xi}})\dif t+\tilde\mE\left[\overline{\overline{c\cdot\p_\nu\Phi}}(\bar X_t^{s,\xi},\cL_{\bar X_t^{s,\xi}})(\tilde{\bar X}_t^{s,\tilde\xi})\right]\dif t\\
&\quad+\overline{\sigma_1^*\cdot\p_x\p_y\Phi}(\bar X_t^{s,\xi},\cL_{\bar X_t^{s,\xi}})\cdot G(\bar X_t^{s,\xi},\cL_{\bar X_t^{s,\xi}})\dif t\\
&\quad+\p_xG(\bar X_t^{s,\xi},\cL_{\bar X_t^{s,\xi}})\bar Z_t^{s,\xi,\vartheta}\dif W_t^1\\
&\quad+\tilde\mE\big[\p_\mu G(\bar X_t^{s,\xi},\cL_{\bar X_t^{s,\xi}})(\tilde{\bar X}_t^{s,\tilde\xi})\tilde{\bar Z}_t^{s,\tilde\xi,\tilde\vartheta}\big]\dif W_t^1\\
&\quad+\overline{\p_y\Phi\cdot\sigma_1}(\bar X_t^{s,\xi},\cL_{\bar X_t^{s,\xi}})\dif W_t^1+\sqrt{\Sigma(\bar X_t^{s,\xi},\cL_{\bar X_t^{s,\xi}})}\dif \tilde W_t,\qquad\bar Z_s^{s,\xi,\vartheta}=\vartheta,
\end{align*}
where the positive semi-definite matrix $\Sigma(x,\mu)$ is given by (\ref{sigma}). For  fixed $T>0$ and  function $\varphi: \sP_2(\mR^{d_1})\to\mR$, we consider the following Cauchy problem on $[0,T]\times\sP_2(\mR^{d_1})\times\sP_2(\mR^{d_1})$ :
\begin{equation} \label{cp1}
\left\{ \begin{aligned}
&\p_t u(t,\cL_\xi,\cL_\vartheta)+\mE\big[\bar F(\xi,\cL_{\xi})\cdot\p_\mu u(t,\cL_\xi,\cL_\vartheta)(\xi)\big]\\
&\quad+\frac{1}{2}\mE\Big[\T\Big(GG^*(\xi,\cL_\xi)\cdot\p_{x}\big[\p_\mu u(t,\cL_\xi,\cL_\vartheta)(\xi)\big]\Big)\Big]\\
&\quad+\mE\bigg[\Big(\p_x\bar F(\xi,\cL_\xi)\vartheta+\tilde\mE\big[\p_\mu\bar F(\xi,\cL_\xi)(\tilde\xi)\tilde\vartheta\big]+\overline{c\cdot\p_y\Phi}(\xi,\cL_\xi)
+\tilde\mE\Big[\overline{\overline{c\cdot\p_\nu\Phi}}(\xi,\cL_{\xi})(\tilde\xi)\Big]\\
&\qquad\quad+\overline{\sigma_1^*\cdot\p_x\p_y\Phi}(\xi,\cL_\xi)\cdot G(\xi,\cL_\xi)\Big)\cdot\p_\pi u(t,\cL_\xi,\cL_\vartheta)(\vartheta)\bigg]\\
&\quad+\frac{1}{2}\mE\bigg[\T\Big(\big[\p_xG(\xi,\cL_\xi)\vartheta+\tilde\mE[\p_\mu G(\xi,\cL_\xi)(\tilde\xi)\tilde\vartheta]\big]\big[\p_xG(\xi,\cL_\xi)\vartheta+\tilde\mE[\p_\mu G(\xi,\cL_\xi)(\tilde\xi)\tilde\vartheta]\big]^*\\
&\qquad\quad\times\p_{z}\big[\p_\pi u(t,\cL_{\xi},\cL_\vartheta)(\vartheta)\big]\Big)\bigg]\\
&\quad+\mE\bigg[\T\Big(\overline{\p_y\Phi\cdot\sigma_1}(\xi,\cL_\xi)\big[\p_xG(\xi,\cL_\xi)\vartheta+\tilde\mE[\p_\mu G(\xi,\cL_\xi)(\tilde\xi)\tilde\vartheta]\big]^*\\
&\qquad\quad\times\p_{z}\big[\p_\pi u(t,\cL_{\xi},\cL_\vartheta)(\vartheta)\big]\Big)\bigg]\\
&\quad+\mE\bigg[\T\Big(\overline{\delta F\cdot\Phi^*}(\xi,\cL_\xi)
\cdot\p_{z}\big[\p_\pi u(t,\cL_{\xi},\cL_\vartheta)(\vartheta)\big]\Big)\bigg]=0,\\
&u(T,\cL_\xi,\cL_\vartheta)=\varphi(\cL_\vartheta).
\end{aligned} \right.
\end{equation}
Then by Lemma \ref{reh}, we have $\bar F\in \big(C_b^{4,(1,3)}\cap C_b^{4,(2,2)}\cap C_b^{4,(3,1)}\big)(\mR^{d_1}\times\sP_2(\mR^{d_1}))$, $\overline{c\cdot\p_y\Phi}$, $\overline{\sigma_1^*\cdot\p_x\p_y\Phi}, \overline{\p_y\Phi\cdot\sigma_1}, \overline{\delta F\cdot\Phi^*}\in C_b^{3,(2,2)}(\mR^{d_1}\times\sP_2(\mR^{d_1}))$ and $\overline{\overline{c\cdot\p_\nu\Phi}}(x,\mu)(\tilde x)\in C_b^{3,(2,2),3}(\mR^{d_1}\times\sP_2(\mR^{d_1})\times\mR^{d_1})$. Therefore, there exists a unique solution $u\in C_b^{1,(2,1),(3,1)}([0,T]\times\sP_2(\mR^{d_1})\times\sP_2(\mR^{d_1}))$ to the equation (\ref{cp1}) by \cite[Theorem 7.2]{BLPR}, which is given by
\begin{align}\label{cp2}
u(t,\cL_\xi,\cL_\vartheta):=\varphi(\cL_{\bar Z^{t,\xi,\vartheta}_T}).
\end{align}

Now, we are in the position to give:

\begin{proof}
Let $u(t,\cL_\xi,\cL_\vartheta)$ be defined by (\ref{cp2}). Then we have
\begin{align*}
\sK(\eps):=\varphi(\cL_{Z_T^\eps})-\varphi(\cL_{\bar Z_T})=u(T,\cL_{X_T^\eps},\cL_{Z_T^\eps})-u(0,\cL_\xi,\delta_0).
\end{align*}
By It\^o's formula, we derive
\begin{align}\label{J1}
\sK(\eps)&=\mE\bigg(\int_0^T\p_tu(t,\cL_{X_t^\eps},\cL_{Z_t^\eps}) +F(X^{\eps}_t,\cL_{X_t^\eps},Y^{\eps}_t,\cL_{Y^{\eps}_t})\cdot\p_{\mu}u(t,\cL_{X_t^\eps},\cL_{Z_t^\eps})(X^{\eps}_t)\no\\ &\quad+\frac{1}{2}\T\Big(GG^*(X^{\eps}_t,\cL_{X_t^\eps})\cdot\p_{x}\big[\p_{\mu}u(t,\cL_{X_t^\eps},\cL_{Z_t^\eps})(X^{\eps}_t)\big]\Big)\dif t\bigg)\no\\
&\quad+\frac{1}{\eps}\mE\bigg(\int_0^T\big[F(X^{\eps}_t,\cL_{X_t^\eps},Y^{\eps}_t,\cL_{Y^{\eps}_t})-\bar F(X^{\eps}_t,\cL_{X_t^\eps})\big]\cdot\p_\pi u(t,\cL_{X_t^\eps},\cL_{Z_t^\eps})(Z^{\eps}_t)\dif t\bigg)\no\\
&\quad+\mE\bigg(\int_0^T\frac{\bar F(X^{\eps}_t,\cL_{X_t^\eps})-\bar F(\bar X_t,\cL_{\bar X_t})}{\eps} \cdot\p_\pi u(t,\cL_{X_t^\eps},\cL_{Z_t^\eps})(Z^{\eps}_t)\dif t\bigg)\no\\
&\quad+\frac{1}{2}\mE\bigg(\int_0^T\T\bigg(\frac{[G(X^{\eps}_t,\cL_{X_t^\eps})-G(\bar X_t,\cL_{\bar X_t})][G(X^{\eps}_t,\cL_{X_t^\eps})-G(\bar X_t,\cL_{\bar X_t})]^*}{\eps^2}\no\\
&\qquad\quad\times\p_z\big[\p_\pi u(t,\cL_{X_t^\eps},\cL_{Z_t^\eps})(Z^{\eps}_t)\big]\bigg)\dif t\bigg).
\end{align}
Note that by the definition of $\bar F(x,\mu)$, the function
$$
[F(x,\mu,y,\nu)-\bar F(x,\mu)]\cdot\p_\pi u(t,\mu,\pi)(z)
$$
satisfies the centering condition (\ref{cen1}). Recall that $\Phi(x,\mu,y,\nu)$ solves the Poisson equation (\ref{pss2}). Then we define
\begin{align*}
\tilde \Phi (t,x,\mu,y,\nu,z,\pi)=\Phi(x,\mu,y,\nu)\cdot\p_\pi u(t,\mu,\pi)(z),
\end{align*}
and get
\begin{align*}
\cL_0(\mu,y,\nu)\tilde \Phi (t,x,\mu,y,\nu,z,\pi)=-[F(x,\mu,y,\nu)-\bar F(x,\mu)]\cdot\p_\pi u(t,\mu,\pi)(z).
\end{align*}
Moreover, we have $[F-\bar F]\cdot\p_\pi u\in C_b^{1,2,(1,1),2,(1,1),2,(1,1)}$ and $\p_xF(x,\mu,\cdot,\nu)\in C_b^1(\mR^{d_2})$. Consequently, it follows by (\ref{es1}) that
\begin{align*}
&\frac{1}{\eps}\mE\left(\int_0^T[F(X_t^\eps,\cL_{X_t^\eps},Y_t^\eps,\cL_{Y_t^\eps})-\bar F(X_t^\eps,\cL_{X_t^\eps})]\cdot\p_\pi u(t,\cL_{X_t^\eps},\cL_{Z_t^\eps})(Z_t^\eps)\dif t\right)\\
&\leq C_0\,\eps+C_0\bigg[\mE\bigg(\int_0^T
\big(\sL_2+\sL_3\big)
\tilde\Phi(t,X_t^\eps,\cL_{X_t^\eps},Y_t^\eps,\cL_{Y_t^\eps},Z_t^\eps,\cL_{Z_t^\eps})\dif t\bigg)\no\\
&\quad+\mE\bigg(\int_0^T\T\big((G\sigma_1^*)(X_t^\eps,\cL_{X_t^\eps},Y_t^\eps,\cL_{Y_t^\eps})
\cdot\p_x\p_y\tilde\Phi(t,X_t^\eps,\cL_{X_t^\eps},Y_t^\eps,\cL_{Y_t^\eps},Z_t^\eps,\cL_{Z_t^\eps})\big)\dif t\bigg)\no\\
&\quad+\mE\bigg(\int_0^T\T\bigg(\sigma_1(\cL_{X_t^\eps},Y_t^\eps,\cL_{Y_t^\eps})\frac{[G(X_t^\eps,\cL_{X_t^\eps})- G(\bar X_t,\cL_{\bar X_t})]^*}{\eps}\no\\
&\qquad\qquad\times\p_y\p_z\tilde\Phi(t,X_t^\eps,\cL_{X_t^\eps},Y_t^\eps,\cL_{Y_t^\eps},Z_t^\eps,\cL_{Z_t^\eps})\bigg)\dif t\bigg)\no\\
&\quad+\mE\tilde\mE\bigg(\int_0^Tc(\tilde X^{\eps}_t,\cL_{X_t^\eps},\tilde Y^{\eps}_t,\cL_{Y^{\eps}_t})\cdot\p_\nu\tilde\Phi(t,X_t^\eps,\cL_{X_t^\eps},Y_t^\eps,\cL_{Y_t^\eps},Z_t^\eps,\cL_{Z_t^\eps})(\tilde Y^{\eps}_t)\dif t\bigg)\\
&\quad+\mE\tilde\mE\bigg(\int_0^T\delta F(\tilde X^{\eps}_t,\cL_{X_t^\eps},\tilde Y^{\eps}_t,\cL_{Y^{\eps}_t})\cdot\p_\pi\tilde\Phi(t,X_t^\eps,\cL_{X_t^\eps},Y_t^\eps,\cL_{Y_t^\eps},Z_t^\eps,\cL_{Z_t^\eps})(\tilde Z^{\eps}_t)\dif t\bigg)\bigg].
\end{align*}
This together with (\ref{J1}) yields
\begin{align}\label{jes}
&\sK(\eps)
\leq C_0\,\eps+\mE\bigg(\int_0^T\p_tu(t,\cL_{X_t^\eps},\cL_{Z_t^\eps})
+F(X^{\eps}_t,\cL_{X_t^\eps},Y^{\eps}_t,\cL_{Y^{\eps}_t})\cdot\p_{\mu}u(t,\cL_{X_t^\eps},\cL_{Z_t^\eps})(X^{\eps}_t)\no\\ &+\frac{1}{2}\T\Big(GG^*(X^{\eps}_t,\cL_{X_t^\eps})\cdot\p_{x}\big[\p_{\mu}u(t,\cL_{X_t^\eps},\cL_{Z_t^\eps})(X^{\eps}_t)\big]\Big)\dif t\bigg)\no\\
&+\mE\bigg(\int_0^T\frac{\bar F(X^{\eps}_t,\cL_{X_t^\eps})-\bar F(\bar X_t,\cL_{\bar X_t})}{\eps} \cdot\p_\pi u(t,\cL_{X_t^\eps},\cL_{Z_t^\eps})(Z^{\eps}_t)\dif t\bigg)\no\\
&+\frac{1}{2}\mE\bigg(\int_0^T\T\bigg(\frac{[G(X^{\eps}_t,\cL_{X_t^\eps})-G(\bar X_t,\cL_{\bar X_t})][G(X^{\eps}_t,\cL_{X_t^\eps})-G(\bar X_t,\cL_{\bar X_t})]^*}{\eps^2}\no\\
&\qquad\quad\times\p_z\big[\p_\pi u(t,\cL_{X_t^\eps},\cL_{Z_t^\eps})(Z^{\eps}_t)\big]\bigg)\dif t\bigg)\no\\
&+C_0\mE\bigg(\int_0^T
\big(\sL_2+\sL_3\big)
\tilde\Phi(t,X_t^\eps,\cL_{X_t^\eps},Y_t^\eps,\cL_{Y_t^\eps},Z_t^\eps,\cL_{Z_t^\eps})\dif t\bigg)\no\\
&+C_0\mE\bigg(\int_0^T\T\big((G\sigma_1^*)(X_t^\eps,\cL_{X_t^\eps},Y_t^\eps,\cL_{Y_t^\eps})
\cdot\p_x\p_y\tilde\Phi(t,X_t^\eps,\cL_{X_t^\eps},Y_t^\eps,\cL_{Y_t^\eps},Z_t^\eps,\cL_{Z_t^\eps})\big)\dif t\bigg)\no\\
&+C_0\mE\bigg(\int_0^T\T\bigg(\sigma_1(\cL_{X_t^\eps},Y_t^\eps,\cL_{Y_t^\eps})\frac{[G(X_t^\eps,\cL_{X_t^\eps})- G(\bar X_t,\cL_{\bar X_t})]^*}{\eps}\no\\
&\qquad\qquad\times\p_y\p_z\tilde\Phi(t,X_t^\eps,\cL_{X_t^\eps},Y_t^\eps,\cL_{Y_t^\eps},Z_t^\eps,\cL_{Z_t^\eps})\bigg)\dif t\bigg)\no\\
&+C_0\mE\tilde\mE\bigg(\int_0^Tc(\tilde X^{\eps}_t,\cL_{X_t^\eps},\tilde Y^{\eps}_t,\cL_{Y^{\eps}_t})\cdot\p_\nu\tilde\Phi(t,X_t^\eps,\cL_{X_t^\eps},Y_t^\eps,\cL_{Y_t^\eps},Z_t^\eps,\cL_{Z_t^\eps})(\tilde Y^{\eps}_t)\dif t\bigg)\no\\
&+C_0\mE\tilde\mE\bigg(\int_0^T\delta F(\tilde X^{\eps}_t,\cL_{X_t^\eps},\tilde Y^{\eps}_t,\cL_{Y^{\eps}_t})\cdot\p_\pi\tilde\Phi(t,X_t^\eps,\cL_{X_t^\eps},Y_t^\eps,\cL_{Y_t^\eps},Z_t^\eps,\cL_{Z_t^\eps})(\tilde Z^{\eps}_t)\dif t\bigg).
\end{align}
For the last term on the right hand side of the above inequality, it is easy to see that for every fixed $(x,y,z)\in\mR^{d_1}\times\mR^{d_2}\times\mR^{d_1}$, the function
$$
(t,\tilde x,\mu,\tilde y,\nu,\tilde z,\pi)\mapsto \delta F(\tilde x,\mu,\tilde y,\nu)\cdot\p_\pi\tilde\Phi(t,x,\mu,y,\nu,z,\pi)(\tilde z)
$$
satisfies the centering condition. Thus by (\ref{es2}) we have
\begin{align*}
&\mE\tilde\mE\bigg(\int_0^T\delta F(\tilde X^{\eps}_t,\cL_{X_t^\eps},\tilde Y^{\eps}_t,\cL_{Y^{\eps}_t})\cdot\p_\pi\tilde\Phi(t,X_t^\eps,\cL_{X_t^\eps},Y_t^\eps,\cL_{Y_t^\eps},Z_t^\eps,\cL_{Z_t^\eps})(\tilde Z^{\eps}_t)\dif t\bigg)\leq C_0\,\eps.
\end{align*}
Substituting this into (\ref{jes}) and in view of (\ref{cp1}), we derive
\begin{align*}
&\sK(\eps)\\
&\leq C_0\,\eps+\mE\bigg(\int_0^T\big[F(X^{\eps}_t,\cL_{X_t^\eps},Y^{\eps}_t,\cL_{Y^{\eps}_t})-\bar F(X^{\eps}_t,\cL_{X_t^\eps})\big]\cdot\p_{\mu}u(t,\cL_{X_t^\eps},\cL_{Z_t^\eps})(X^{\eps}_t)\dif t\bigg)\\
&\quad+\mE\bigg(\int_0^T\bigg[\frac{\bar F(X_t^\eps,\cL_{X_t^\eps})-\bar F(\bar X_t,\cL_{\bar X_t})}{\eps}-\p_x\bar F(X_t^\eps,\cL_{X_t^\eps})Z_t^\eps-\tilde\mE\big[\p_\mu\bar F(X_t^\eps,\cL_{X_t^\eps})(\tilde X_t^\eps)\tilde Z_t^\eps\big]\bigg]\\
&\qquad\qquad\times\p_\pi u(t,\cL_{X_t^\eps},\cL_{Z_t^\eps})(Z_t^\eps)\dif t\bigg)\\
&\quad+C_0\mE\bigg(\int_0^T
\big[c\cdot\p_y\Phi(X_t^\eps,\cL_{X_t^\eps},Y_t^\eps,\cL_{Y_t^\eps})-\overline{c\cdot\p_y\Phi}(X_t^\eps,\cL_{X_t^\eps})\big]
\cdot\p_\pi u(t,\cL_{X_t^\eps},\cL_{Z_t^\eps})(Z_t^\eps)\dif t\bigg)\\
&\quad+C_0\mE\tilde\mE\bigg(\int_0^T\big[c(\tilde X^{\eps}_t,\cL_{X_t^\eps},\tilde Y^{\eps}_t,\cL_{Y^{\eps}_t})\cdot\p_\nu\Phi(X_t^\eps,\cL_{X_t^\eps},Y_t^\eps,\cL_{Y_t^\eps})(\tilde Y_t^\eps
)-\overline{\overline{c\cdot\p_\nu\Phi}}(X_t^\eps,\cL_{X_t^\eps})(\tilde X_t^\eps)\big]\\
&\qquad\qquad\times\p_\pi u(t,\cL_{X_t^\eps},\cL_{Z_t^\eps})(Z^{\eps}_t)\dif t\bigg)\\
&\quad+C_0\mE\bigg(\int_0^T\big[(G\sigma_1^*)(X_t^\eps,\cL_{X_t^\eps},Y_t^\eps,\cL_{Y_t^\eps})
\cdot\p_x\p_y\Phi(X_t^\eps,\cL_{X_t^\eps},Y_t^\eps,\cL_{Y_t^\eps})\\
&\qquad\qquad-\overline{(G\sigma_1^*)\cdot\p_x\p_y\Phi}(X_t^\eps,\cL_{X_t^\eps})\big]\cdot\p_\pi u(t,\cL_{X_t^\eps},\cL_{Z_t^\eps})(Z_t^\eps)\dif t\bigg)\no\\
&\quad+\frac{1}{2}\mE\bigg(\int_0^T\T\bigg(\bigg[\frac{[G(X^{\eps}_t,\cL_{X_t^\eps})-G(\bar X_t,\cL_{\bar X_t})][G(X^{\eps}_t,\cL_{X_t^\eps})-G(\bar X_t,\cL_{\bar X_t})]^*}{\eps^2}\\
&\qquad\qquad-\big[\p_xG(X_t^\eps,\cL_{X_t^\eps})Z_t^\eps+\tilde\mE[\p_\mu G(X_t^\eps,\cL_{X_t^\eps})(\tilde X_t^\eps)\tilde Z_t^\eps]\big]\big[\p_xG(X_t^\eps,\cL_{X_t^\eps})Z_t^\eps\\
&\qquad\qquad+\tilde\mE[\p_\mu G(X_t^\eps,\cL_{X_t^\eps})(\tilde X_t^\eps)\tilde Z_t^\eps]\big]^*\bigg]\cdot\p_z\big[\p_\pi u(t,\cL_{X_t^\eps},\cL_{Z_t^\eps})(Z^{\eps}_t)\big]\bigg)\dif t\bigg)\\
&\quad+C_0\mE\bigg(\int_0^T\T\Big(\big[\delta F\cdot\Phi^*(X_t^\eps,\cL_{X_t^\eps},Y_t^\eps,\cL_{Y_t^\eps})\\
&\qquad\qquad-\overline{ \delta F\cdot\Phi^*}(X_t^\eps,\cL_{X_t^\eps})\big]\cdot\p_z\big[\p_\pi u(t,\cL_{X_t^\eps},\cL_{Z_t^\eps})(Z_t^\eps)\big]\Big)
\dif t\bigg)\no\\
&\quad+C_0\mE\bigg(\int_0^T\T\bigg(\Big[\p_y\Phi(X_t^\eps,\cL_{X_t^\eps},Y_t^\eps,\cL_{Y_t^\eps})
\cdot\sigma_1(\cL_{X_t^\eps},Y_t^\eps,\cL_{Y_t^\eps})\frac{[G(X_t^\eps,\cL_{X_t^\eps})- G(\bar X_t,\cL_{\bar X_t})]^*}{\eps}\\
&\qquad\qquad-\overline{\p_y\Phi\cdot\sigma_1}(X_t^\eps,\cL_{X_t^\eps})\big[\p_xG(X_t^\eps,\cL_{X_t^\eps})Z_t^\eps
+\tilde\mE[\p_\mu G(X_t^\eps,\cL_{X_t^\eps})(\tilde X_t^\eps)\tilde Z_t^\eps]\big]^*\Big]\\
&\qquad\qquad\times\p_z\big[\p_\pi u(t,\cL_{X_t^\eps},\cL_{Z_t^\eps})(Z_t^\eps)\big]\bigg)\dif t\bigg)=:C_0\,\eps+\sum_{i=1}^8\sK_i(\eps).
\end{align*}
Note that the function
$$
(t,x,\mu,y,\nu,\pi)\mapsto [F(x,\mu,y,\nu)-\bar F(x,\mu)]\cdot\p_\mu u(t,\mu,\pi)(x)
$$
satisfies the centering condition (\ref{cen1}) and belongs to $C_b^{1,2,(1,1),2,(1,1),(1,1)}$.  Then as a direct result of the estimate (\ref{es2}), we have
$$
\sK_1(\eps)\leq C_1\,\eps.
$$
Similarly, according to (\ref{bc1})-(\ref{sigma2}) and the estimate (\ref{es2}), we also have
\begin{align*}
\sK_3(\eps)+\sK_4(\eps)+\sK_5(\eps)+\sK_7(\eps)\leq C_2\,\eps.
\end{align*}
In what follows, we estimate the remaining three terms one by one. Let us first handle $\sK_2(\eps)$ and $\sK_6(\eps)$. Using the mean value theorem and Theorem \ref{main1}, we deduce that
\begin{align*}
&\sK_2(\eps)+\sK_6(\eps)\\
&\leq C_3\int_0^T\mE\bigg|\frac{\bar F(X_t^\eps,\cL_{X_t^\eps})-\bar F(\bar X_t,\cL_{\bar X_t})}{\eps}-\p_x\bar F(X_t^\eps,\cL_{X_t^\eps})Z_t^\eps-\tilde\mE\big[\p_\mu\bar F(X_t^\eps,\cL_{X_t^\eps})(\tilde X_t^\eps)\tilde Z_t^\eps\big]\bigg|\dif t\\
&\quad+C_3\int_0^T\mE\bigg\|\frac{[G(X^{\eps}_t,\cL_{X_t^\eps})-G(\bar X_t,\cL_{\bar X_t})][G(X^{\eps}_t,\cL_{X_t^\eps})-G(\bar X_t,\cL_{\bar X_t})]^*}{\eps^2}\\
&\qquad\qquad-\big[\p_xG(X_t^\eps,\cL_{X_t^\eps})Z_t^\eps+\tilde\mE[\p_\mu G(X_t^\eps,\cL_{X_t^\eps})(\tilde X_t^\eps)\tilde Z_t^\eps]\big]\big[\p_xG(X_t^\eps,\cL_{X_t^\eps})Z_t^\eps\\
&\qquad\qquad+\tilde\mE[\p_\mu G(X_t^\eps,\cL_{X_t^\eps})(\tilde X_t^\eps)\tilde Z_t^\eps]\big]^*\bigg\|\dif t\\
&\leq C_3\int_0^T\big(\mE|X_t^\eps-\bar X_t|^2\big)^\frac{1}{2}\cdot\big(1+\mE|Z_t^\eps|^4\big)\dif t\leq C_3\,\eps.
\end{align*}
As for $\sK_8(\eps)$, using the same technique as above, by the mean value theorem and (\ref{es2}) again, we get that
\begin{align*}
\sK_8(\eps)
&\leq C_4\mE\bigg(\int_0^T\bigg\|\p_y\Phi(X_t^\eps,\cL_{X_t^\eps},Y_t^\eps,\cL_{Y_t^\eps})
\cdot\sigma_1(\cL_{X_t^\eps},Y_t^\eps,\cL_{Y_t^\eps})\bigg[\frac{[G(X_t^\eps,\cL_{X_t^\eps})- G(\bar X_t,\cL_{\bar X_t})]^*}{\eps}\\
&\qquad\quad-\big[\p_xG(X_t^\eps,\cL_{X_t^\eps})Z_t^\eps
+\tilde\mE[\p_\mu G(X_t^\eps,\cL_{X_t^\eps})(\tilde X_t^\eps)\tilde Z_t^\eps]\big]^*\bigg]\bigg\|\dif t\bigg)\\
&\quad+C_4\mE\bigg(\int_0^T\T\bigg(\Big[\p_y\Phi(X_t^\eps,\cL_{X_t^\eps},Y_t^\eps,\cL_{Y_t^\eps})
\cdot\sigma_1(\cL_{X_t^\eps},Y_t^\eps,\cL_{Y_t^\eps})-\overline{\p_y\Phi\cdot\sigma_1}(X_t^\eps,\cL_{X_t^\eps})\Big]\\
&\qquad\qquad\quad\times\big[\p_xG(X_t^\eps,\cL_{X_t^\eps})Z_t^\eps
+\tilde\mE[\p_\mu G(X_t^\eps,\cL_{X_t^\eps})(\tilde X_t^\eps)\tilde Z_t^\eps]\big]^*\\
&\qquad\qquad\quad\times\p_z\big[\p_\pi u(t,\cL_{X_t^\eps},\cL_{Z_t^\eps})(Z_t^\eps)\big]\bigg)\dif t\bigg)\\
&\leq C_4\int_0^T\big(\mE|X_t^\eps-\bar X_t|^2\big)^\frac{1}{2}\cdot\big(1+\mE|Y_t^\eps|^4+\mE|Z_t^\eps|^4\big)\dif t+C_4\,\eps\leq C_4\,\eps.
\end{align*}
Combining the above computations, we arrive at the desired conclusion. Thus the proof is completed.
\end{proof}

\section{Appendix}

In this section, we give an It\^{o} formula under the case of common noises. Let $X_t$ and $Y_t$ are two $d$-dimensional It\^{o} processes, i.e.,
\begin{align*}
&\dif X_t=f(t)\dif t+g_1(t)\dif W_t+g_2(t)\dif B_t, \qquad X_0=\xi,\\
&\dif Y_t=b(t)\dif t+\sigma_1(t)\dif W_t+\sigma_2(t)\dif B_t, \qquad\, Y_0=\eta,
\end{align*}
where $W_t$ and $B_t$ are two independent Brownian motion, and $f(t),b(t), g_i(t),\sigma_i(t)(i=1,2)$ are progressively measurable processes with respect to the filtration $\{\sF_t\}_{t\geq0}$, such that for every $T>0$,
\begin{align*}
\mathbb{E}\Big[\int_{0}^{T}(|f(t)|^{2}+\|g_1(t)\|^{4}+\|g_2(t)\|^{4}+|b(t)|^{2}+\|\sigma_1(t)\|^{4}+\|\sigma_2(t)\|^{4})\dif t\Big]<\infty.
\end{align*}
Then we have:

\begin{lemma}\label{ito}
Assume that $u\in C^{2,(1,1),2,(1,1)}
(\mR^{d}\times\sP_{2}(\mR^{d})\times\mR^{d}\times\sP_{2}(\mR^{d}))$ and for every compact set $\cK\subset \mR^{d}\times\sP_{2}(\mR^{d})\times\mR^{d}\times\sP_{2}(\mR^{d})$,
\begin{align*}
\sup_{(x,\mu,y,\nu)\in\cK}&\bigg[\int_{\mR^d}\Big(|\p_\mu u(x,\mu,y,\nu)(\tilde x)|^2+\|\p_{\tilde x}\p_\mu u(x,\mu,y,\nu)(\tilde x)\|^2\Big)\mu(\dif \tilde x)\\
&+\int_{\mR^d}\Big(|\p_\nu u(x,\mu,y,\nu)(\tilde y)|^2+\|\p_{\tilde y}\p_\nu u(x,\mu,y,\nu)(\tilde y)\|^2\Big)\nu(\dif \tilde y)\bigg]<\infty.
\end{align*}
Then for $\mu_t=\cL_{X_t}$ and $\nu_t=\cL_{Y_t}$, we have $\mathbb{P}$-a.s.,
\begin{align}\label{ito1}
&u(X_t,\mu_t,Y_t,\nu_t)=u(\xi,\cL_{\xi},\eta,\cL_{\eta})+\int_0^tf(s)\cdot\p_xu(X_s,\mu_s,Y_s,\nu_s)\dif s\no\\
&\quad+\frac{1}{2}\int_0^t\T\big([g_1g_1^*+g_2g_2^*](s)\cdot\p^2_xu(X_s,\mu_s,Y_s,\nu_s)\big)\dif s\no\\
&\quad+\int_0^t\p_xu(X_s,\mu_s,Y_s,\nu_s)\cdot g_1(s)\dif W_s\no+\int_0^t\p_xu(X_s,\mu_s,Y_s,\nu_s)\cdot g_2(s)\dif B_s\no\\
&\quad+\int_0^tb(s)\cdot\p_yu(X_s,\mu_s,Y_s,\nu_s)\dif s+\frac{1}{2}\int_0^t\T\big([\sigma_1\sigma_1^*+\sigma_2\sigma_2^*](s)\cdot\p^2_yu(X_s,\mu_s,Y_s,\nu_s)\big)\dif s\no\\
&\quad+\int_0^t\p_yu(X_s,\mu_s,Y_s,\nu_s)\cdot \sigma_1(s)\dif W_s\no+\int_0^t\p_yu(X_s,\mu_s,Y_s,\nu_s)\cdot \sigma_2(s)\dif B_s\no\\
&\quad+\int_0^t\T\big([g_1\sigma_1^*+g_2\sigma_2^*](s)\cdot\p_y\p_xu(X_s,\mu_s,Y_s,\nu_s)\big)\dif s\no\\
&\quad+\int_0^t\tilde\mE\big[\tilde f(s)\cdot\p_\mu u(X_s,\mu_s,Y_s,\nu_s)(\tilde X_s)\big]\dif s\no\\
&\quad+\frac{1}{2}\int_0^t\tilde \mE\big[\T\big([\tilde g_1\tilde g_1^*+\tilde g_2\tilde g_2^*](s)\cdot\p_{\tilde x}\p_\mu u(X_s,\mu_s,Y_s,\nu_s)(\tilde X_s)\big)\big]\dif s\no\\
&\quad+\int_0^t\tilde \mE\big[\tilde b(s)\cdot\p_\nu u(X_s,\mu_s,Y_s,\nu_s)(\tilde Y_s)\big]\dif s\no\\
&\quad+\frac{1}{2}\int_0^t\tilde \mE\big[\T\big([\tilde \sigma_1\tilde \sigma_1^*+\tilde \sigma_2\tilde \sigma_2^*](s)\cdot\p_{\tilde y}\p_\nu u(X_s,\mu_s,Y_s,\nu_s)(\tilde Y_s)\big)\big]\dif s,
\end{align}
where the process $(\tilde X_t,\tilde f(t), \tilde g_1(t),\tilde g_2(t),\tilde Y_t,\tilde b(t), \tilde \sigma_1(t),\tilde \sigma_2(t))$ is a copy of the original process $(X_t,f(t),g_1(t),g_2(t),Y_t,b(t),\sigma_1(t),\sigma_2(t))$ defined on a
copy $(\tilde \Omega,\tilde\sF,\tilde \mP)$ of the original probability space $(\Omega,\sF,\mP)$.
\end{lemma}

\begin{proof}
Following the idea of the proof of It\^{o}'s formula in \cite[Proposition 5.102]{CD}, we assume without loss of generality that the derivatives of $u$ are bounded, and $f(t), b(t),g_i(t),\sigma_i(t)(i=1,2)$ are continuous and satisfy
\begin{align*}
\mathbb{E}\Big[\sup_{0\leq t\leq T}(|f(t)|^{2}+\|g_1(t)\|^{4}+\|g_2(t)\|^{4}+|b(t)|^{2}+\|\sigma_1(t)\|^{4}+\|\sigma_2(t)\|^{4})\Big]<\infty.
\end{align*}
For every $k\geq1$, let $(X_t^k,Y_t^k)$ be independent copies of $(X_t,Y_t)$, namely,
\begin{align*}
&\dif X^k_t=f^k(t)\dif t+g_1^k(t)\dif W^k_t+g_2^k(t)\dif B^k_t,~~ X^k_0=\xi^k,\\
&\dif Y^k_t=b^k(t)\dif t+\sigma_1^k(t)\dif W^k_t+\sigma_2^k(t)\dif B^k_t,~~ Y^k_0=\eta^k.
\end{align*}
Denote by $\mu_t^N:=\frac{1}{N}\sum_{k=1}^N\delta_{X^k_t}$ and $\nu_t^N:=\frac{1}{N}\sum_{k=1}^N\delta_{Y^k_t}$ the empirical measures of $(X^k_t)_{1\leq k\leq N}$ and $(Y^k_t)_{1\leq k\leq N}$, respectively. Then we define
\begin{align*}
u^N(x_1,\cdots,x_N,y_1,\cdots,y_N)=u\bigg(\frac{1}{N}\sum_{k=1}^N\delta_{x_k},\frac{1}{N}\sum_{k=1}^N\delta_{y_k}\bigg).
\end{align*}
Using the classical It\^{o}'s formula, we derive
\begin{align*}
&u(\mu_t^N,\nu_t^N)=u^N(X^1_t,\cdots,X^N_t,Y^1_t,\cdots,Y^N_t)\\
&=u^N(\xi^1,\cdots,\xi^N,\eta^1,\cdots,\eta^N)
+\int_0^t\frac{1}{N}\sum_{k=1}^N\p_\mu u(\mu_s^N,\nu_s^N)(X^k_s)\cdot f^k(s)\dif s\\
&\quad+\frac{1}{N}\sum_{k=1}^N\int_0^t\p_\mu u(\mu_s^N,\nu_s^N)(X^k_s)\cdot g_1^k(s)\dif W_s^k+\frac{1}{N}\sum_{k=1}^N\int_0^t\p_\mu u(\mu_s^N,\nu_s^N)(X^k_s)\cdot g_2^k(s)\dif B_s^k\\
&\quad+\frac{1}{2N}\int_0^t\sum_{k=1}^N\T\big([g_1^kg_1^{k,*}+g_2^kg_2^{k,*}](s)\cdot\p_{\tilde x}\p_\mu u(\mu_s^N,\nu_s^N)(X^k_s)\big)\dif s\\
&\quad+\frac{1}{2N^2}\int_0^t\sum_{k=1}^N\T\big([g_1^kg_1^{k,*}+g_2^kg_2^{k,*}](s)\cdot\p^2_\mu u(\mu_s^N,\nu_s^N)(X^k_s,X^k_s)\big)\dif s\\
&\quad+\int_0^t\frac{1}{N}\sum_{k=1}^N\p_\nu u(\mu_s^N,\nu_s^N)(Y^k_s)\cdot b^k(s)\dif s\\
&\quad+\frac{1}{N}\sum_{k=1}^N\int_0^t\p_\nu u(\mu_s^N,\nu_s^N)(Y^k_s)\cdot \sigma_1^k(s)\dif W_s^k+\frac{1}{N}\sum_{k=1}^N\int_0^t\p_\nu u(\mu_s^N,\nu_s^N)(Y^k_s)\cdot \sigma_2^k(s)\dif B_s^k\\
&\quad+\frac{1}{2N}\int_0^t\sum_{k=1}^N\T\big([\sigma_1^k\sigma_1^{k,*}+\sigma_2^k\sigma_2^{k,*}](s)\cdot\p_{\tilde y}\p_\nu u(\mu_s^N,\nu_s^N)(Y^k_s)\big)\dif s\\
&\quad+\frac{1}{2N^2}\int_0^t\sum_{k=1}^N\T\big([\sigma_1^k\sigma_1^{k,*}+\sigma_2^k\sigma_2^{k,*}](s)\cdot\p^2_\nu u(\mu_s^N,\nu_s^N)(Y^k_s,Y^k_s)\big)\dif s\\
&\quad+\frac{1}{N^2}\int_0^t\sum_{k=1}^N\T\big([g_1^k\sigma_1^{k,*}+g_2^k\sigma_2^{k,*}](s)\cdot\p_\nu\p_\mu u(\mu_s^N,\nu_s^N)(X^k_s,Y^k_s)\big)\dif s.
\end{align*}
Taking expectation from both sides of the above equality (the stochastic
integrals have zero expectation due to the properties of martingale) gives
\begin{align*}
\mE u(\mu_t^N,\nu_t^N)
&=\mE u(\mu_0^N,\nu_0^N)
+\int_0^t\mE\big[\p_\mu u(\mu_s^N,\nu_s^N)(X^1_s)\cdot f^1(s)\big]\dif s\\
&\quad+\frac{1}{2}\int_0^t\mE\big[\T\big([g_1^1g_1^{1,*}+g_2^1g_2^{1,*}](s)\cdot\p_{\tilde x}\p_\mu u(\mu_s^N,\nu_s^N)(X^1_s)\big)\big]\dif s\\
&\quad+\frac{1}{2N}\int_0^t\mE\big[\T\big([g_1^1g_1^{1,*}+g_2^1g_2^{1,*}](s)\cdot\p^2_\mu u(\mu_s^N,\nu_s^N)(X^1_s,X^1_s)\big)\big]\dif s\\
&\quad+\int_0^t\mE\big[\p_\nu u(\mu_s^N,\nu_s^N)(Y^1_s)\cdot b^1(s)\big]\dif s\\
&\quad+\frac{1}{2}\int_0^t\mE\big[\T\big([\sigma_1^1\sigma_1^{1,*}+\sigma_2^1\sigma_2^{1,*}](s)\cdot\p_{\tilde y}\p_\nu u(\mu_s^N,\nu_s^N)(Y^1_s)\big)\big]\dif s\\
&\quad+\frac{1}{2N}\int_0^t\mE\big[\T\big([\sigma_1^1\sigma_1^{1,*}+\sigma_2^1\sigma_2^{1,*}](s)\cdot\p^2_\nu u(\mu_s^N,\nu_s^N)(Y^1_s,Y^1_s)\big)\big]\dif s\\
&\quad+\frac{1}{N}\int_0^t\mE\big[\T\big([g_1^1\sigma_1^{1,*}+g_2^1\sigma_2^{1,*}](s)\cdot\p_\nu\p_\mu u(\mu_s^N,\nu_s^N)(X^1_s,Y^1_s)\big)\big]\dif s,
\end{align*}
where we have used that $(Y_t^k,X^k_t)_{1\leq k\leq N}$ are independently identically distributed. We know from \cite{CD} that $\mP$-a.s.
\begin{align*}
\cW_2(\mu_t^N,\mu_t)\to0,~~\cW_2(\nu_t^N,\nu_t)\to0, ~~~as~N\to\infty.
\end{align*}
Taking limit $N\to\infty$ yields
\begin{align}\label{ito2}
&u(\mu_t,\nu_t)
=u(\cL_{\xi},\cL_{\eta})+\int_0^t\mE\big[\p_\mu u(\mu_s,\nu_s)(X_s)\cdot f(s)\big]\dif s\no\\
&\quad+\frac{1}{2}\int_0^t\mE\big[\T\big([g_1g_1^*+g_2g_2^*](s)\cdot\p_{\tilde x}\p_\mu u(\mu_s,\nu_s)(X_s)\big)\big]\dif s\no\\
&\quad+\int_0^t\mE\big[\p_\nu u(\mu_s,\nu_s)(Y_s)\cdot b(s)\big]\dif s+\frac{1}{2}\int_0^t\mE\big[\T\big([\sigma_1\sigma_1^*+\sigma_2\sigma_2^*](s)\cdot\p_{\tilde y}\p_\nu u(\mu_s,\nu_s)(Y_s)\big)\big]\dif s.
\end{align}
Then we define the function:
$$
U(t,x,y)=u(x,\mu_t,y,\nu_t).
$$
For any $t\geq0$ and $h>0$, by (\ref{ito2}) we compute
\begin{align*}
&U(t+h,x,y)-U(t,x,y)=u(x,\mu_{t+h},y,\nu_{t+h})-u(x,\mu_t,y,\nu_t)\\
&=\int_t^{t+h}\mE\big[\p_\mu u(x,\mu_s,y,\nu_s)(X_s)\cdot f(s)\big]\dif s\no\\
&\quad+\frac{1}{2}\int_t^{t+h}\mE\big[\T\big([g_1g_1^*+g_2g_2^*](s)\cdot\p_{\tilde x}\p_\mu u(x,\mu_s,y,\nu_s)(X_s)\big)\big]\dif s\no\\
&\quad+\int_t^{t+h}\mE\big[\p_\nu u(x,\mu_s,y,\nu_s)(Y_s)\cdot b(s)\big]\dif s\\
&\quad+\frac{1}{2}\int_t^{t+h}\mE\big[\T\big([\sigma_1\sigma_1^*+\sigma_2\sigma_2^*](s)\cdot\p_{\tilde y}\p_\nu u(x,\mu_s,y,\nu_s)(Y_s)\big)\big]\dif s.
\end{align*}
Consequently, we have that $U$ is differentiable in $t$ and
\begin{align}\label{ito3}
\p_tU(t,x,y)&=\mE\big[\p_\mu u(x,\mu_t,y,\nu_t)(X_t)\cdot f(t)\big]+\mE\big[\p_\nu u(x,\mu_t,y,\nu_t)(Y_t)\cdot b(t)\big]\no\\
&\quad+\frac{1}{2}\mE\big[\T\big([g_1g_1^*+g_2g_2^*](t)\cdot\p_{\tilde x}\p_\mu u(x,\mu_t,y,\nu_t)(X_t)\big)\big]\no\\
&\quad+\frac{1}{2}\mE\big[\T\big([\sigma_1\sigma_1^*+\sigma_2\sigma_2^*](t)\cdot\p_{\tilde y}\p_\nu u(x,\mu_t,y,\nu_t)(Y_t)\big)\big].
\end{align}
Applying the classical It\^{o} formula to $U$, we get
\begin{align*}
U(t,X_t,Y_t)&=U(0,\xi,\eta)+\int_0^t\p_tU(s,X_s,Y_s)\dif s+\int_0^tf(s)\cdot\p_x U(s,X_s,Y_s)\dif s\\
&\quad+\frac{1}{2}\int_0^t\T\big([g_1g_1^*+g_2g_2^*](s)\cdot\p^2_xU(s,X_s,Y_s)\big)\dif s\no\\
&\quad+\int_0^t\p_xU(s,X_s,Y_s)\cdot g_1(s)\dif W_s\no+\int_0^t\p_xU(s,X_s,Y_s)\cdot g_2(s)\dif B_s\no\\
&\quad+\int_0^tb(s)\cdot\p_yU(s,X_s,Y_s)\dif s+\frac{1}{2}\int_0^t\T\big([\sigma_1\sigma_1^*+\sigma_2\sigma_2^*](s)\cdot\p^2_yU(s,X_s,Y_s)\big)\dif s\no\\
&\quad+\int_0^t\p_yU(s,X_s,Y_s)\cdot \sigma_1(s)\dif W_s\no+\int_0^t\p_yU(s,X_s,Y_s)\cdot \sigma_2(s)\dif B_s\no\\
&\quad+\int_0^t\T\big([g_1\sigma_1^*+g_2\sigma_2^*](s)\cdot\p_y\p_xU(s,X_s,Y_s)\big)\dif s,
\end{align*}
which together with (\ref{ito3}) implies the desired result.
\end{proof}


\bigskip
\begin{thebibliography}{2}

\bibitem{BH}J. Bao and X. Huang: Approximations of Mckean-Vlasov SDEs with irregular coefficients. arXiv:1905.08522.

\bibitem{BYY} J. Bao, G. Yin and C. Yuan:   Two-time-scale stochastic partial differential equations driven by $\alpha$-stable noises: Averaging principles. {\it Bernoulli}, {\bf  23} (2017), 645--669.

\bibitem{BSY}J. Bao, M. Scheutzow and C. Yuan: Existence of invariant probability measures for functional McKean-Vlasov SDEs. {\it Electron. J. Probab.}, {\bf 27} (2022), 1--14.



\bibitem{BR}V. Barbu and M. R\"ockner: From non-linear Fokker-Planck equations to solutions of distribution dependent SDE. {\it Ann. Probab.}, {\bf 48} (2020), 1902--1920.
	

\bibitem{BS1}
    Z. W. Bezemek and K. Spiliopoulos:  Large deviations for interacting multiscale particle systems. arXiv: 2011.03032.

\bibitem{BS2}
	Z. W. Bezemek and K. Spiliopoulos: Rate of homogenization for fully-coupled McKean-Vlasov SDEs. arXiv:2202.07753v1.

\bibitem{BS3}
    Z. W. Bezemek and K. Spiliopoulos: Moderate deviations for fully coupled multiscale weakly interacting particle systems. arXiv: 2202.08403.

\bibitem{BLPR}
   R. Buckdahn, J. Li, S. Peng and C. Rainer: Mean-field stochastic differential equations and associated PDEs. {\it Ann. Probab.}, {\bf 45} (2017), 824--787.




\bibitem{C} P. Cardaliaguet P.:  Notes on mean field games. https:/\!/www.ceremade.dauphine.fr/cardaliaguet/MF-G20130420.pdf, 2013.

\bibitem{CD}
	R. Carmona and F. Delarue: Probabilistic Theory of Mean Field Games with Applications I: Mean Field FBSDEs, Control, and Games, Probability Theory and Stochastic Modelling, Springer, 2018.



\bibitem{CF}
    P.-E. Chaudru de Raynal and N. Frikha: Well-Posedness for some non-linear SDEs and related PDE on the Wassertein space. {\it J. Math. Pures Appl.}, {\bf 159} (2022), 1--167.

\bibitem{CM}
    D. Crisan and E. McMurray: Smoothing properties of McKean-Vlasov SDEs. {\it Probab. Theory and Related Fields}, {\bf 171} (2018),  97--148.

\bibitem{DGP}
     M. G. Delgadino, R. S. Gvalani and G. A. Pavliotis: On the diffusive-mean field limit for weakly interacting diffusions exhibiting phase transitions.  {\it Arch. Rational Mech. Anal.}, {\bf 241} (2021), 91--148.

\bibitem{GP}
     S. N. Gomes and G. A. Pavliotis: Mean field limits for interacting diffusions in a two-scale
	potential. {\it J. Non-linear Sci.}, {\bf 28} (2018), 905-941.



\bibitem{HLL1}
	W. Hong, S. Li and W. Liu: Strong convergence rates in averaging principle for slow-fast McKean-Vlasov SPDEs. {\it J. Differential Equations}, {\bf 316} (2022), 94--135.

\bibitem{HLL2}
	W. Hong, S. Li, W. Liu and X. Sun: Central limit type theorem and large deviations for multi-scale McKean-Vlasov SDEs. arXiv:2112.08203v1.

\bibitem{HW}
	X. Huang and F.-Y. Wang: Distribution dependent SDEs with singular coefficients. {\it Stochastic Process. Appl.}, {\bf129} (2019),  4747--4770.

\bibitem{HL}
    W. Hu and C. J. Li: A convergence analysis of the perturbed compositional gradient flow: averaging principle
and normal deviations. {\it Discrete Cont. Dynam. Syst.-A}, {\bf 38} (2018), 4951--4977.

\bibitem{K}	
   M. Kac: Foundations of kinetic theory. In Proceedings of the Third Berkeley Symposium on Mathematical Statistics and Probability: Contributions to Astronomy and Physics, pages 171--197, Berkeley, Calif., 1956. University of California Press.



\bibitem{Lions}
P. Lions: Mean-field games and applications. Lectures at the College de France, 2007.

\bibitem{L}
   D. Liu: Strong convergence of principle of averaging for multiscale stochastic dynamical systems. {\it Commun. Math. Sci.}, {\bf 8} (2010), 999--1020.


\bibitem{LWX}
    Y. Li, F. Wu and L. Xie: Poisson equation on Wasserstein space and diffusion approximations for McKean-Vlasov equation. arXiv:2203.12796.

\bibitem{M}
	H. P. McKean, Jr.: A class of Markov processes associated with nonlinear parabolic equations. {\it Proc. Nat. Acad. Sci. USA}, {\bf56 } (1966), 1907--1911.

\bibitem{MV}	Y. S. Mishura and A. Yu. Veretennikov: Existence and uniqueness theorems for solutions of McKean-Vlasov stochastic equations. {\it Theory Probab. Math. Stat.}, {\bf 103} (2021), 59--101.

%
%


\bibitem{PS}
    G. A. Pavliotis and A. M. Stuart: Multiscale methods: averaging and homogenization. Texts Appl. Math., vol. 53. Springer, New York, 2008.



\bibitem{RSX}
	M. R\"{o}ckner, X. Sun and Y. Xie: Strong convergence order for slow-fast McKean-Vlasov stochastic differential equations. {\it Ann. Inst. H. Poincar\'{e} Probab. Statist.}, {\bf57} (2021), 547--576.

	
\bibitem{RX}
    M. R\"ockner and L. Xie: Averaging principle and normal deviations for multiscale stochastic systems. {\it Commun. Math. Phys.}, {\bf 383} (2021), 1889--1937.

\bibitem{RXY}M. R\"ockner, L. Xie and L. Yang: Averaging principle and normal deviations for multi-scale stochastic hyperbolic-parabolic equations. {\it Stoch. Partial Differ. Equ. Anal. Comput.}, (2022), doi.org/10.1007/s40072-022-00248-8.

\bibitem{RZ}
    M. R\"ockner and X. Zhang: Well-posedness of distribution dependent SDEs with singular drifts. {\it Bernoulli}, {\bf 27} (2021), 1131--1158.

\bibitem{SY}Y. Suo and C. Yuan: CLT and MDP for McKean-Vlasov SDEs. arXiv:1910.04418.

\bibitem{S}
    A.-S. Sznitman, {\sl Topics in propagation of chaos}. In: Hennequin, PL. (eds), Ecole d'Et\'{e} de Probabilit\'{e}s de Saint-Flour XIX - 1989. Lecture Notes in Math, Vol 1464. Springer, Berlin, 1991, 165--251.

\bibitem{S1}
     K. Spiliopoulos: Fluctuation analysis and short time asymptotics for multiple scales diffusion processes.
{\it Stoch. Dyna.}, {\bf 14} (2014),  1350026.

\bibitem{W1}
    F.-Y. Wang: Distribution dependent SDEs for Landau type equations. {\it Stoch. Proc. Appl.}, {\bf 128} (2018), 595--621.

\bibitem{WR}
    W. Wang and A.J. Roberts: Average and deviation for slow-fast stochastic partial differential equations, {\it J. Differential Equations}, {\bf 253} (2012), 1265--1286.	

\bibitem{X}L. Xie:    Fast-slow stochastic dynamical system with singular coefficients. {\it Science China Math.}, doi.org/10.1007/s11425-020-1971-1.
	
\end{thebibliography}
\end{document}